\documentclass[a4paper,12pt]{amsart}
\usepackage{mathrsfs}
\usepackage{amssymb}
\usepackage{fancyhdr}
\usepackage{typearea}
\usepackage{color}
\usepackage{amsmath,amstext,amsthm,amscd, hyperref, verbatim, array}

\usepackage[a4paper,top=2.3cm,bottom=2.3cm,left=2cm,right=2cm]{geometry}

\makeatletter
\def\@setcopyright{}
\def\serieslogo@{}
\makeatother

\newcommand{\Complex}{\mathbb C}
\newcommand{\Real}{\mathbb R}
\newcommand{\N}{\mathbb N}
\newcommand{\ddbar}{\overline\partial}
\newcommand{\pr}{\partial}
\newcommand{\ol}{\overline}
\newcommand{\Td}{\widetilde}
\newcommand{\norm}[1]{\left\Vert#1\right\Vert}
\newcommand{\abs}[1]{\left\vert#1\right\vert}

\newcommand{\set}[1]{\left\{#1\right\}}
\newcommand{\To}{\rightarrow}
\DeclareMathOperator{\Ker}{Ker}

\theoremstyle{plain}
\newtheorem{theorem}{Theorem}[section]
\newtheorem{lemma}[theorem]{Lemma}
\newtheorem{proposition}[theorem]{Proposition}
\newtheorem{corollary}[theorem]{Corollary}
\newtheorem{definition}[theorem]{Definition}

\theoremstyle{definition}

\theoremstyle{remark}
\newtheorem{remark}[theorem]{Remark}
\numberwithin{equation}{section}

\begin{document}

\title[Morse inequalities on CR covering manifolds with $S^1$-action]
{Morse inequalities for Fourier components of Kohn-Rossi cohomology of CR covering manifolds with $S^1$-action}

\author{Rung-Tzung Huang}
\address{Department of Mathematics, National Central University, Chung-Li 320, Taiwan}
\email{rthuang@math.ncu.edu.tw}

\author{Guokuan Shao}
\address{School of Mathematics (Zhuhai), Sun Yat-sen University, Zhuhai 519082, Guangdong, China}
\thanks{
The first author was supported by Taiwan Ministry of Science and Technology project 107-2115-M-008-007-MY2. Both authors would like to express their gratitude to Prof. Chin-Yu Hsiao for very helpful comments in this work.}
\email{shaogk@mail.sysu.edu.cn}

\keywords{Kohn-Rossi cohomology, heat kernel, CR manifold} 
\subjclass[2010]{58J35, 32V20}

\begin{abstract}
Let $X$ be a compact connected CR manifold of dimension $2n+1, n \geq 1$. Let $\widetilde{X}$ be a paracompact CR manifold with a transversal CR $S^1$-action, such that there is a discrete group $\Gamma$ acting freely on $\widetilde{X}$ having $X \, = \, \widetilde{X}/\Gamma$. Based on an asymptotic formula for the Fourier components of the heat kernel with respect to the $S^1$-action, we establish the Morse inequalities for Fourier components of reduced $L^2$-Kohn-Rossi cohomology with values in a rigid CR vector bundle over $\widetilde{X}$. As a corollary, we obtain the Morse inequalities for Fourier components of Kohn-Rossi cohomology on $X$ which were obtained by Hsiao-Li \cite{HsiaoLi15} by using Szeg\"{o} kernel method. %We give a proof by the heat kernel method.
\end{abstract}

\maketitle
%%%%%%%%%%%%%%%%%%%%%%%%%%%%%%%%%%%%%%%%%%%%%%%%%%

\section{Introduction and statement of the results}

Gromov-Henkin-Shubin \cite[Theorem 0.2]{GHS98} considered covering manifolds that are strongly pseudoconvex of complex manifolds and analyzed the holomorphic $L^2$-functions on the coverings. Todor-Chiose-Marinescu \cite{TCM} generalized in a similar manner the Morse inequalities of Siu-Demailly \cite{S84, D85} on coverings of complex manifolds, they also considered coverings of
weakly pseudoconvex domains in \cite{MTC}.
 The study of problems on CR manifolds with $S^{1}$-action becomes active recently, see \cite{CHT, Hsiao14, HH16, HL1, HsiaoLi15,HS} and the references therein. In particular, Hsiao-Li \cite{HsiaoLi15} established the Morse inequalities for Fourier components of Kohn-Rossi cohomology on $X$ by using the Szeg\"{o} kernel method. In \cite{HM} general Morse inequalities
for CR bundles are proved, generalizing \cite{Ge}.
Inspired by the results of \cite{GHS98, HsiaoLi15, TCM, S84, D85}, we establish Morse inequalities for Fourier components of reduced $L^2$-Kohn-Rossi cohomology with values in a rigid CR vector bundle on a covering manifold over a compact connected CR manifold with $S^{1}$-action. This generalizes the results of \cite{HsiaoLi15} to CR covering manifolds with $S^1$-action. We present a proof by the heat kernel method, which is inspired by Bismut's proof \cite{B, MM} of the holomorphic Morse inequalities. The crucial estimate for Fourier components of the heat kernel of Kohn Laplacians was given in \cite{HH16}.

Now we formulate the main results. We refer to other sections for notations and definitions (see Definition \ref{d-gue160502}, \ref{d-gue150508f}, \ref{d-gue150508ff}, \ref{d-gue150508dI} and \eqref{e-KR}, \eqref{e-L2KR}) used here. Let $X$ be a compact connected CR manifold of dimension $2n+1, n \geq 1$ with a transversal CR $S^1$-action $e^{i\theta}$ on $X$.
For $x\in X$, we say that the period of $x$ is $\frac{2\pi}{\ell}$, $\ell\in\mathbb N$, if $e^{i\theta}\circ x\neq x$, for every $0<\theta<\frac{2\pi}{\ell}$, and $e^{i\frac{2\pi}{\ell}}\circ x=x$. For each $\ell\in\mathbb N$, put 
\begin{equation}\label{e-gue150802bm}
X_\ell=\set{x\in X;\, \mbox{the period of $x$ is $\frac{2\pi}{\ell}$}}
\end{equation} 
and let 
\begin{equation}\label{e-gue15080202}
p=\min\set{\ell\in\mathbb N;\, X_\ell\neq\emptyset}.
\end{equation}
It is well-known that if $X$ is connected, then $X_p$ is an open and dense subset of $X$ (see Duistermaat-Heckman~\cite[p.265]{Du82}). Assume $X \, = \, X_{p_1} \cup X_{p_2} \cup \cdots \cup X_{p_k},$ $p=: p_1 < p_2 < \cdots < p_k.$
Set $X_{{\rm reg\,}}:=X_{p}$. 
We call $x\in X_{{\rm reg\,}}$ a regular point of the $S^1$ action. Let $X_{{\rm sing\,}}$ be the complement of $X_{{\rm reg\,}}$.

Let $\widetilde{X}$ be a paracompact CR manifold, such that there is a discrete group $\Gamma$ acting freely on $\widetilde{X}$ having $X \, = \, \widetilde{X}/\Gamma$. Let $\pi \, : \, \widetilde{X} \to X$ be the natural projection with the pull-back map $\pi^* : TX \to T\widetilde{X}$. Then $\widetilde{X}$ admits a pull-back CR structure $T^{1,0}\widetilde{X} := \pi^*T^{1,0}X$ and, hence, a CR manifold. We assume that $\widetilde{X}$ admits a transversal CR locally free $S^1$ action, denote by $e^{i\theta}$. We further assume that the map 
 $$
 \Gamma \times \widetilde{X} \to \widetilde{X}, \  (\gamma, \widetilde{x}) \mapsto \gamma \circ \widetilde{x}, \quad \forall \widetilde{x} \in \widetilde{X}, \quad \forall \gamma \in \Gamma.
 $$
 is CR, see \eqref{E:crmap}, and 
 $$
 e^{i\theta} \circ \gamma \circ \widetilde{x} \, = \, \gamma \circ e^{i\theta} \circ \widetilde{x}, \quad \forall \widetilde{x} \in \widetilde{X}, \quad \forall \theta \in [0, 2\pi[, \quad \forall \gamma \in \Gamma.
 $$ 
 Let $\widetilde{E}:=\pi^* E$ be the pull-pack bundle of a rigid CR vector bundle $E$ over $X$. Then $\widetilde{E}$ is a $\Gamma$-invariant rigid CR vector bundle over $\widetilde{X}$. We denote by $\widetilde{X}_{{\rm reg\,}}$ the set of regular points of the $S^1$-action on $\widetilde{X}$. Note that since $\Gamma$ acts on $\widetilde{X}$ freely so that $\widetilde{X}/\Gamma = X$, hence, we have $\widetilde{X}_{{\rm reg\,}}/\Gamma \, = \, X_{{\rm reg\,}} \, = \, X_p$. We denote by $X(q)$ a subset of $X$ such that 
 \[
 X(q) \, := \, \left\{ x \in X : \text{ $\mathcal{L}_x$ has exactly $q$ negative eigenvalues and $n-q$ positive eigenvalues} \right\}.
 \]
We refer to Section \ref{S:prelim} for more details.
Our main theorem is the following
\begin{theorem}\label{t-01covering}

	With the above notations and assumptions, 
	as $m \to \infty$, for $q=0, 1, \cdots, n,$ the $m$-th Fourier components of reduced $L^2$-Kohn-Rossi cohomology (see \eqref{e-L2KR}) satisfy the following strong Morse inequalities
	\begin{equation}\label{e-101covering}
	\begin{split}
	&\sum_{j=0}^q(-1)^{q-j}\dim_\Gamma \overline{H}^{j}_{b, (2), m}(\widetilde{X},\widetilde{E})\\
	&\, \le \, \frac{prm^{n}}{2\pi^{n+1}} \sum_{j=0}^q (-1)^{q-j} \int_{X(j)} \left| \det(\mathcal{L}_{x}) \right| dv_{X}(x)
	+o(m^{n}), \ \ \text{for} \ \  p\mid m,\\
	&\sum_{j=0}^{q}(-1)^{q-j}\dim_\Gamma \overline{H}^{j}_{b, (2), m}(\widetilde{X},\widetilde{E})=o(m^{n}), \ \ \text{for} \ \  p\nmid m.
	\end{split}
	\end{equation}
	where $r$ denotes the rank of $\widetilde{E}$, $\dim_\Gamma$ denotes the Von Neumann dimension (see \S \ref{SS:covering} in the below, \cite[\S 3.6.1]{MM} or \cite[\S 3]{A76}) and $\mathcal{L}_{x}$ is the Levi form at $x\in X$.
	When $p\mid m$, $q=n$, as $m \to \infty$, we have the asymptotic Riemann-Roch-Hirzebruch theorem
	\begin{equation}\label{e-102covering1}
	\sum_{j=0}^n(-1)^{j}\dim_\Gamma \overline{H}^{j}_{b, (2), m}(\widetilde{X},\widetilde{E})
	\, = \, \frac{prm^{n}}{2\pi^{n+1}} \sum_{j=0}^n (-1)^{j} \int_{X(j)} \left| \det(\mathcal{L}_{x}) \right| dv_{X}(x)
	+o(m^{n}). 
	\end{equation}
	In particular, we get the weak Morse inequalities
	\begin{equation}\label{e-102covering}
	\dim_\Gamma \overline{H}^{q}_{b, (2), m}(\widetilde{X},\widetilde{E})
	\, \le \, \frac{prm^{n}}{2\pi^{n+1}}\int_{X(q)} \left| \det(\mathcal{L}_{x}) \right| dv_{X}(x)
	+o(m^{n}).
	\end{equation}
       \end{theorem}

By the standard argument in \cite{HsiaoLi15} or \cite{HHLS}, we deduce easily the following Grauert-Riemenschneider criterion on coverings of CR manifolds.
\begin{corollary}\label{c-912}
With the above notations and assumptions in Theorem \ref{t-01covering}, we assume also that $X$ is weakly pseudoconvex and strongly pseudoconvex at a point. Then
\begin{equation}\label{e-912}
\dim_\Gamma \overline{H}^{0}_{b, (2), m}(\widetilde{X},\widetilde{E})\approx m^n,
\ \ \text{for} \ \ p\mid m.
\end{equation}
In particular, $\dim_\Gamma \overline{H}^{0}_{b, (2)}(\widetilde{X},\widetilde{E})=\infty$.
\end{corollary}

When $\Gamma = \left\{ e \right\}$, $p=1$ and $\widetilde{E}$ is trivial line bundle, we deduce the following Morse inequalities of Hsiao-Li, see \cite[Theorem 2.2 and Theorem 2.5]{HsiaoLi15}. 

\begin{corollary}\label{t-01}
With the above notations and assumptions, 
	as $m \to \infty$, for $q=0, 1, \cdots, n,$ the $m$-th Fourier components of Kohn-Rossi cohomology satisfy the following strong Morse inequalities, 
	\begin{equation}\label{e-101}
	\sum_{j=0}^q(-1)^{q-j}\dim H^{j}_{b, m}(X)
	\, \le \, \frac{m^{n}}{2\pi^{n+1}} \sum_{j=0}^q (-1)^{q-j} \int_{X(j)} \left| \det(\mathcal{L}_{x}) \right| dv_{X}(x)
	+o(m^{n}), 
	\end{equation}
	where $\mathcal{L}_{x}$ is the Levi form at $x\in X$.
	In particular, we get the weak Morse inequalities
	\begin{equation}\label{e-102}
\dim H^{q}_{b, m}(X)
\, \le \, \frac{m^{n}}{2\pi^{n+1}}\int_{X(q)} \left| \det(\mathcal{L}_{x}) \right| dv_{X}(x)
	+o(m^{n}). 
	\end{equation}
\end{corollary}

Let $X$ be a compact CR manifold of dimension $2n+1$, $n\geq 1$. A classical theorem due to Boutet de Monvel \cite{BM75} asserts that $X$ can be globally CR embedded into $\Complex^{N}$, for some $N\in\N$, when $X$ is strongly pseudoconvex with dimension $n\geq 5$.
Epstein \cite{Ep92} proved that if $X$ is strongly pseudoconvex with dimension $3$ and a global free transversal CR $S^{1}$-action, then $X$ can be embedded into $\Complex^{N}$ by positive Fourier components of CR functions. Corollary \ref{t-01} guarantees the abundance of positive Fourier components of CR functions to do embedding in general cases (e.g. the $S^{1}$-action can be only locally free). In \cite{HsiaoLi15}, the authors' proofs include localization of analytic objects (eigenfunctions, Szeg\"{o} kernels), Kohn $L^{2}$ estimates and scaling techniques. A more general version of Corollary \ref{t-01} (with $X$ being weakly pseudoconvex) is proved by Cheng-Hsiao-Tsai in \cite[Proposition 1.20 and Corollary 1.21]{CHT} in a different way. By using the Morse inequalities, \eqref{e-101} and \eqref{e-102}, Hsiao-Li \cite[Theorem 2.6]{HsiaoLi15} proved that there are abundant CR functions on $X$ when $X$ is weakly pseudoconvex and strongly pseudoconvex at a point. Corollary \ref{c-912} generalizes Theorem 2.6 of \cite{HsiaoLi15} to CR covering manifolds.  

This paper is organized as follows. In Section \ref{S:prelim} we introduce some basic notations, terminology and  definitions. In Section \ref{S:asympexp} we study the asymptotic behavior of heat kernels of Kohn Laplacians.
Section \ref{S:proof} is devoted to the heat kernel proof of the main theorem. 

\section{Preliminaries}\label{S:prelim}
\subsection{Some standard notations}\label{s-gue150508b}
We use the following notations: $\mathbb N=\set{1,2,\ldots}$, $\mathbb N_0=\mathbb N\cup\set{0}$, $\Real$ 
is the set of real numbers, $\Real_+:=\set{x\in\Real;\, x>0}$, $\ol\Real_+:=\set{x\in\Real;\, x\geq0}$.
For a multiindex $\alpha=(\alpha_1,\ldots,\alpha_n)\in\mathbb N_0^n$
we set $\abs{\alpha}=\alpha_1+\cdots+\alpha_n$. For $x=(x_1,\ldots,x_n)$ we write
\[
\begin{split}
&x^\alpha=x_1^{\alpha_1}\ldots x^{\alpha_n}_n,\quad
\pr_{x_j}=\frac{\pr}{\pr x_j}\,,\quad
\pr^\alpha_x=\pr^{\alpha_1}_{x_1}\ldots\pr^{\alpha_n}_{x_n}=\frac{\pr^{\abs{\alpha}}}{\pr x^\alpha}\,.
\end{split}
\]
Let $z=(z_1,\ldots,z_n)$, $z_j=x_{2j-1}+ix_{2j}$, $j=1,\ldots,n$, be coordinates of $\Complex^n$.
We write
\[
\begin{split}
&z^\alpha=z_1^{\alpha_1}\ldots z^{\alpha_n}_n\,,\quad\ol z^\alpha=\ol z_1^{\alpha_1}\ldots\ol z^{\alpha_n}_n\,,\\
&\pr_{z_j}=\frac{\pr}{\pr z_j}=
\frac{1}{2}\Big(\frac{\pr}{\pr x_{2j-1}}-i\frac{\pr}{\pr x_{2j}}\Big)\,,\quad\pr_{\ol z_j}=
\frac{\pr}{\pr\ol z_j}=\frac{1}{2}\Big(\frac{\pr}{\pr x_{2j-1}}+i\frac{\pr}{\pr x_{2j}}\Big),\\
&\pr^\alpha_z=\pr^{\alpha_1}_{z_1}\ldots\pr^{\alpha_n}_{z_n}=\frac{\pr^{\abs{\alpha}}}{\pr z^\alpha}\,,\quad
\pr^\alpha_{\ol z}=\pr^{\alpha_1}_{\ol z_1}\ldots\pr^{\alpha_n}_{\ol z_n}=
\frac{\pr^{\abs{\alpha}}}{\pr\ol z^\alpha}\,.
\end{split}
\]

Let $X$ be a $C^\infty$ orientable paracompact manifold. 
We let $TX$ and $T^*X$ denote the tangent bundle of $X$ and the cotangent bundle of $X$, respectively.
The complexified tangent bundle of $X$ and the complexified cotangent bundle of $X$ 
will be denoted by $\Complex TX$ and $\Complex T^*X$, respectively. We write $\langle\,\cdot\,,\cdot\,\rangle$ 
to denote the pointwise duality between $T^*X$ and $TX$.
We extend $\langle\,\cdot\,,\cdot\,\rangle$ bilinearly to $\Complex T^*X\times\Complex TX$. For $u\in \Complex T^*X$, $v\in\Complex TX$, we also write $u(v):=\langle\,u\,,v\,\rangle$.

Let $Y\subset X$ be an open set. The spaces of
smooth sections of $E$ over $Y$ and distribution sections of $E$ over $Y$ will be denoted by $C^\infty(Y, E)$ and $D'(Y, E)$, respectively.

\subsection{CR manifolds with $S^1$-action}\label{s-gue150508bI} 

Let $(X, T^{1,0}X)$ be a compact CR manifold of dimension $2n+1$, $n\geq 1$, where $T^{1,0}X$ is a CR structure of $X$. That is, $T^{1,0}X$ is a subbundle of rank $n$ of the complexified tangent bundle $\mathbb{C}TX$, satisfying $T^{1,0}X\cap T^{0,1}X=\{0\}$, where $T^{0,1}X=\overline{T^{1,0}X}$, and $[\mathcal V,\mathcal V]\subset\mathcal V$, where $\mathcal V=C^\infty(X, T^{1,0}X)$. We assume that $X$ admits a $S^1$ action: $S^1\times X\rightarrow X$. We write $e^{i\theta}$ to denote the $S^1$ action. Let $T\in C^\infty(X, TX)$ be the global real vector field induced by the $S^1$ action given by 
$(Tu)(x)=\frac{\partial}{\partial\theta}\left(u(e^{i\theta}\circ x)\right)|_{\theta=0}$, $u\in C^\infty(X)$. 

\begin{definition}\label{d-gue160502}
	We say that the $S^1$ action $e^{i\theta}$ is CR if
	$[T, C^\infty(X, T^{1,0}X)]\subset C^\infty(X, T^{1,0}X)$ and the $S^1$ action is transversal if for each $x\in X$,
	$\Complex T(x)\oplus T_x^{1,0}X\oplus T_x^{0,1}X=\mathbb CT_xX$. Moreover, we say that the $S^1$ action is locally free if $T\neq0$ everywhere. 
\end{definition}

Note that if the $S^1$ action is transversal, then it is locally free.
We assume throughout that $(X, T^{1,0}X)$ is a connected CR manifold with a transversal CR $S^1$ action $e^{i\theta}$ and we let $T$ be the global vector field induced by the $S^1$ action. Let $\omega_0\in C^\infty(X,T^*X)$ be the global real one form determined by $\langle\,\omega_0\,,\,u\,\rangle=0$, for every $u\in T^{1,0}X\oplus T^{0,1}X$ and $\langle\,\omega_0\,,\,T\,\rangle=-1$. 

\begin{definition}\label{d-gue150508f}
	For $p\in X$, the Levi form $\mathcal L_p$ is the Hermitian quadratic form on $T^{1,0}_pX$ given by 
	$\mathcal{L}_p(U,\ol V)=-\frac{1}{2i}\langle\,d\omega_0(p)\,,\,U\wedge\ol V\,\rangle$, $U, V\in T^{1,0}_pX$.
\end{definition}

\begin{definition}\label{d-gue150508ff}
If the Levi form $\mathcal{L}_p$ is positive definite, we say that $X$ is strongly pseudoconvex at $p$. If the Levi form is positive definite at every point of $X$, we say that $X$ is strongly pseudoconvex.
\end{definition}
%We assume throughout that $(X, T^{1,0}X)$ is strongly pseudoconvex. 

Denote by $T^{*1,0}X$ and $T^{*0,1}X$ the dual bundles of
$T^{1,0}X$ and $T^{0,1}X$, respectively. Define the vector bundle of $(0,q)$ forms by
$T^{*0,q}X=\Lambda^q(T^{*0,1}X)$. Put $T^{*0,\bullet}X:=\oplus_{j\in\set{0,1,\ldots,n}}T^{*0,j}X$.
Let $D\subset X$ be an open subset. Let $\Omega^{0,q}(D)$
denote the space of smooth sections of $T^{*0,q}X$ over $D$ and let $\Omega_0^{0,q}(D)$
be the subspace of $\Omega^{0,q}(D)$ whose elements have compact support in $D$. Put 
\[\begin{split}
&\Omega^{0,\bullet}(D):=\oplus_{j\in\set{0,1,\ldots,n}}\Omega^{0,j}(D),\\
&\Omega^{0,\bullet}_0(D):=\oplus_{j\in\set{0,1,\ldots,n}}\Omega^{0,j}_0(D).\end{split}\]
Similarly, if $E$ is a vector bundle over $D$, then we let $\Omega^{0,q}(D,E)$ denote the space of smooth sections of $T^{*0,q}X\otimes E$ over $D$ and let $\Omega_0^{0,q}(D,E)$ be the subspace of $\Omega^{0,q}(D, E)$ whose elements have compact support in $D$. Put 
\[\begin{split}
&\Omega^{0,\bullet}(D,E):=\oplus_{j\in\set{0,1,\ldots,n}}\Omega^{0,j}(D,E),\\
&\Omega^{0,\bullet}_0(D,E):=\oplus_{j\in\set{0,1,\ldots,n}}\Omega^{0,j}_0(D,E).\end{split}\]

Fix $\theta_0\in]-\pi, \pi[$, $\theta_0$ small. Let
$$d e^{i\theta_0}: \mathbb CT_x X\rightarrow \mathbb CT_{e^{i\theta_0}x}X$$
denote the differential map of $e^{i\theta_0}: X\rightarrow X$. By the CR property of the $S^1$ action, we can check that
\begin{equation}\label{e-gue150508fa}
\begin{split}
de^{i\theta_0}:T_x^{1,0}X\rightarrow T^{1,0}_{e^{i\theta_0}x}X,\\
de^{i\theta_0}:T_x^{0,1}X\rightarrow T^{0,1}_{e^{i\theta_0}x}X,\\
de^{i\theta_0}(T(x))=T(e^{i\theta_0}x).
\end{split}
\end{equation}
Let $(e^{i\theta_0})^*:\Lambda^j(\Complex T^*X)\To\Lambda^j(\Complex T^*X)$ be the pull-back map by $e^{i\theta_0}$, $j=0,1,\ldots,2n+1$. From \eqref{e-gue150508fa}, it is easy to see that for every $q=0,1,\ldots,n$, 
\begin{equation}\label{e-gue150508faI}
(e^{i\theta_0})^*:T^{*0,q}_{e^{i\theta_0}x}X\To T^{*0,q}_{x}X.
\end{equation}
Let $u\in\Omega^{0,q}(X)$. Define
\begin{equation}\label{e-gue150508faII}
Tu:=\frac{\pr}{\pr\theta}\bigr((e^{i\theta})^*u\bigr)|_{\theta=0}\in\Omega^{0,q}(X).
\end{equation}
For every $\theta\in\Real$ and every $u\in C^\infty(X,\Lambda^j(\Complex T^*X))$, we write $u(e^{i\theta}\circ x):=(e^{i\theta})^*u(x)$. 

Let $\ddbar_b:\Omega^{0,q}(X)\rightarrow\Omega^{0,q+1}(X)$ be the tangential Cauchy-Riemann operator. From the CR property of the $S^1$ action, it is straightforward to see that
\[T\ddbar_b=\ddbar_bT\ \ \mbox{on $\Omega^{0,\bullet}(X)$}.\]

\begin{definition}\label{d-gue50508d}
	Let $D\subset U$ be an open set. We say that a function $u\in C^\infty(D)$ is rigid if $Tu=0$. We say that a function $u\in C^\infty(X)$ is Cauchy-Riemann (CR for short)
	if $\ddbar_bu=0$. We call $u$ a rigid CR function if  $\ddbar_bu=0$ and $Tu=0$.
\end{definition}

\begin{definition} \label{d-gue150508dI}
	Let $F$ be a complex vector bundle over $X$. We say that $F$ is rigid (CR) if 
	$X$ can be covered with open sets $U_j$ with trivializing frames $\set{f^1_j,f^2_j,\dots,f^r_j}$, $j=1,2,\ldots$, such that the corresponding transition matrices are rigid (CR). The frames $\set{f^1_j,f^2_j,\dots,f^r_j}$, $j=1,2,\ldots$, are called rigid (CR) frames. 
\end{definition}

\begin{definition}\label{d-gue150514f}
	Let $F$ be a complex rigid vector bundle over $X$ and let $\langle\,\cdot\,|\,\cdot\,\rangle_F$ be a Hermitian metric on $F$. We say that $\langle\,\cdot\,|\,\cdot\,\rangle_F$ is a rigid Hermitian metric if for every rigid local frames $f_1,\ldots, f_r$ of $F$, we have $T\langle\,f_j\,|\,f_k\,\rangle_F=0$, for every $j,k=1,2,\ldots,r$. 
\end{definition} 

It is known that there is a rigid Hermitian metric on any rigid vector bundle $F$ (see Theorem 2.10 in~\cite{CHT} and Theorem 10.5 in~\cite{Hsiao14}). Note that  Baouendi-Rothschild-Treves~\cite{BRT85} proved that $T^{1,0}X$ is a rigid complex vector bundle over $X$.

From now on, let $E$ be a rigid CR vector bundle over $X$ and we take a rigid Hermitian metric $\langle\,\cdot\,|\,\cdot\,\rangle_E$ on $E$ and take a rigid Hermitian metric $\langle\,\cdot\,|\,\cdot\,\rangle$ on $\Complex TX$ such that $T^{1,0}X\perp T^{0,1}X$, $T\perp (T^{1,0}X\oplus T^{0,1}X)$, $\langle\,T\,|\,T\,\rangle=1$.
The Hermitian metrics on $\Complex TX$ and on $E$ induce Hermitian metrics $\langle\,\cdot\,|\,\cdot\,\rangle$ and $\langle\,\cdot\,|\,\cdot\,\rangle_E$ on $T^{*0,\bullet}X$ and $T^{*0,\bullet}X\otimes E$, respectively. 
We denote by $dv_X=dv_X(x)$ the volume form on $X$ induced by the fixed 
Hermitian metric $\langle\,\cdot\,|\,\cdot\,\rangle$ on $\Complex TX$. Then we get natural global $L^2$ inner products $(\,\cdot\,|\,\cdot\,)_{E}$, $(\,\cdot\,|\,\cdot\,)$
on $\Omega^{0,\bullet}(X,E)$ and $\Omega^{0,\bullet}(X)$, respectively. We denote by $L^2(X,T^{*0,q}X
\otimes E)$ and $L^2(X,T^{*0,q}X)$ the completions of $\Omega^{0,q}(X,E)$ and $\Omega^{0,q}(X)$ with respect to $(\,\cdot\,|\,\cdot\,)_{E}$ and $(\,\cdot\,|\,\cdot\,)$, respectively. Similarly, we denote by $L^2(X,T^{*0,\bullet}X
\otimes E)$ and $L^2(X,T^{*0,\bullet}X)$ the completions of $\Omega^{0,\bullet}(X,E)$ and $\Omega^{0,\bullet}(X)$ with respect to $(\,\cdot\,|\,\cdot\,)_{E}$ and $(\,\cdot\,|\,\cdot\,)$, respectively. We extend $(\,\cdot\,|\,\cdot\,)_{E}$ and $(\,\cdot\,|\,\cdot\,)$ to $L^2(X,T^{*0,\bullet}X\otimes E)$ and $L^2(X,T^{*0,\bullet}X)$ in the standard way, respectively.  For $f\in L^{2}(X,T^{*0,\bullet}X\otimes E)$, we denote $\norm{f}^2_{E}:=(\,f\,|\,f\,)_{E}$. Similarly, for $f\in L^{2}(X,T^{*0,\bullet}X)$, we denote $\norm{f}^2:=(\,f\,|\,f\,)$. 

We also write $\ddbar_b$ to denote the tangential Cauchy-Riemann operator acting on forms with values in $E$:
\[\ddbar_b:\Omega^{0,\bullet}(X, E)\To\Omega^{0,\bullet}(X,E).\]
Since $E$ is rigid, we can also define $Tu$ for every $u\in\Omega^{0,q}(X,E)$ and we have 
\begin{equation}\label{e-gue150508d}
T\ddbar_b=\ddbar_bT\ \ \mbox{on $\Omega^{0,\bullet}(X,E)$}.
\end{equation}
For every $m\in\mathbb Z$, let
\begin{equation}\label{e-gue150508dI}
\begin{split}
&\Omega^{0,q}_m(X,E):=\set{u\in\Omega^{0,q}(X,E);\, Tu=imu},\ \ q=0,1,2,\ldots,n,\\
&\Omega^{0,\bullet}_m(X,E):=\set{u\in\Omega^{0,\bullet}(X,E);\, Tu=imu}.
\end{split}
\end{equation} 
For each $m\in\mathbb Z$, we denote by $L^2_m(X,T^{*0,q}X\otimes E)$ and $L^2_m(X,T^{*0,q}X)$ the completions of $\Omega^{0,q}_m(X,E)$ and $\Omega^{0,q}_m(X)$ with respect to $(\,\cdot\,|\,\cdot\,)_{E}$ and $(\,\cdot\,|\,\cdot\,)$, respectively. Similarly, we denote by $L^2_m(X,T^{*0,\bullet}X\otimes E)$ and $L^2_m(X,T^{*0,\bullet}X)$ the completions of $\Omega^{0,\bullet}_m(X,E)$ and $\Omega^{0,\bullet}_m(X)$ with respect to $(\,\cdot\,|\,\cdot\,)_{E}$ and $(\,\cdot\,|\,\cdot\,)$, respectively.

\subsection{Covering manifolds, Von Neumann dimension}\label{SS:covering} 
Let $(X, T^{1,0}X)$ be a compact CR manifold of dimension $2n+1$, $n\geq 1$. Let $\widetilde{X}$ be a paracompact CR manifold, such that there is a discrete group $\Gamma$ acting freely on $\widetilde{X}$ having $X \, = \, \widetilde{X}/\Gamma$. Let $\pi \, : \, \widetilde{X} \to X$ be the natural projection with the pull-back map $\pi^* : TX \to T\widetilde{X}$. Then $\widetilde{X}$ admits a pull-back CR structure $T^{1,0}\widetilde{X} := \pi^*T^{1,0}X$ and, hence, a CR manifold. We assume that $\widetilde{X}$ admits a transversal CR locally free $S^1$ action, denoted by $e^{i\theta}$. We further assume that the map 
 $$
 \Gamma \times \widetilde{X} \to \widetilde{X}, \  (\gamma, \widetilde{x}) \mapsto \gamma \circ \widetilde{x}, \quad \forall \widetilde{x} \in \widetilde{X}, \quad \forall \gamma \in \Gamma.
 $$
 is CR, i.e. 
\begin{equation}\label{E:crmap}
\gamma_*(T^{1,0}_{\widetilde{x}}\widetilde{X}) \subseteq T^{1,0}_{\gamma \cdot {\widetilde{x}}}\widetilde{X},
\end{equation} 
and 
 $$
 e^{i\theta} \circ \gamma \circ \widetilde{x} \, = \, \gamma \circ e^{i\theta} \circ \widetilde{x}, \quad \forall \widetilde{x} \in \widetilde{X}, \quad \forall \theta \in [0, 2\pi[, \quad \forall \gamma \in \Gamma.
 $$ 
It is easy to see that the $S^1$-action $e^{i\theta}$ on $\widetilde{X}$ induces a transversal CR locally free $S^1$ action, also denoted by $e^{i\theta}$.
We denote by $\widetilde{T}:= \pi^* T$ the pull-back one form on $\widetilde{X}$, then $T$ is the global real vector field induced by the $S^1$-action on $X$. Let $\widetilde{\omega}_0 \, := \, \pi^* \omega_0$ be the pull-back one form on $\widetilde{X}$, where $\omega_0$ is the global real one form on $X$ as defined in Subsection~\ref{s-gue150508bI}. Then, for $\widetilde{p} \in \widetilde{X}$, the Levi form $\widetilde{\mathcal{L}}_{\widetilde{p}}$ is the Hermitian quadratic form on $T^{1,0}_{\widetilde{p}} \widetilde{X}$ given by 
\begin{equation}\label{e-que1807241627}
\widetilde{\mathcal{L}}_{\widetilde{p}}( \widetilde{U},\ol{\widetilde{V}}) = -\frac{1}{2i}\langle\,d \widetilde{\omega}_0(\widetilde{p})\,,\, \widetilde{U} \wedge \ol{\widetilde{V}} \,\rangle =-\frac{1}{2i}\langle\,d\omega_0(\pi(\widetilde{p}))\,,\, \pi_* \widetilde{U} \wedge \pi_* \ol{\widetilde{V}}\,\rangle, 
\end{equation} where $\widetilde{U}, \widetilde{V} \in T^{1,0}_{\widetilde{p}} \widetilde{X}$. 

As usual, let $\Omega^{0,q}(\widetilde{X})$ denote the space of smooth sections of $\wedge^q(T^{*0,1}\widetilde{X})$. We also denote by $\ddbar_b:\Omega^{0,q}(\widetilde{X})\rightarrow\Omega^{0,q+1}(\widetilde{X})$ the tangential Cauchy-Riemann operator. Then $\widetilde{T}\ddbar_b=\ddbar_b\widetilde{T}\ \ \mbox{on $\Omega^{0,\bullet}(\widetilde{X})$}.$ Let $E$ be a rigid CR vector bundle over $X$, then $\widetilde{E} := \pi^*E$ is a $\Gamma$-invariant rigid CR vector bundle over $\widetilde{X}$. Again let $\Omega^{0,q}(\widetilde{X}, \widetilde{E})$ denote the space of smooth sections of $\wedge^q(T^{*0,1}\widetilde{X})\otimes \widetilde{E}$. We again denote by $\ddbar_b:\Omega^{0,q}(\widetilde{X}, \widetilde{E})\rightarrow\Omega^{0,q+1}(\widetilde{X}, \widetilde{E})$ the tangential Cauchy-Riemann operator. Then again
$\widetilde{T}\ddbar_b=\ddbar_b\widetilde{T}\ \ \mbox{on $\Omega^{0,\bullet}(\widetilde{X}, \widetilde{E})$}$. We denote by $L^2(\widetilde{X},T^{*0,q}\widetilde{X}\otimes \widetilde{E})$ and $L^2(\widetilde{X},T^{*0,q}\widetilde{X})$ the completions of $\Omega^{0,q}(\widetilde{X},\widetilde{E})$ and $\Omega^{0,q}(\widetilde{X})$ with respect to the corresponding pull-back metrics $(\,\cdot\,|\,\cdot\,)_{\widetilde{E}}$ and $(\,\cdot\,|\,\cdot\,)$. Similarly, we denote by $L^2(\widetilde{X},T^{*0,\bullet}\widetilde{X}\otimes \widetilde{E})$ and $L^2(\widetilde{X},T^{*0,\bullet}\widetilde{X})$ the completions of $\Omega^{0,\bullet}(\widetilde{X},\widetilde{E})$ and $\Omega^{0,\bullet}(\widetilde{X})$ with respect to the corresponding pull-back metrics $(\,\cdot\,|\,\cdot\,)_{\widetilde{E}}$ and $(\,\cdot\,|\,\cdot\,)$. 

As usual, for every $m\in\mathbb Z$, let
\begin{equation}\label{e-gue180803}
\begin{split}
&\Omega^{0,q}_m(\widetilde{X},\widetilde{E}):=\set{u\in\Omega^{0,q}(\widetilde{X},\widetilde{E});\, \widetilde{T}u=imu},\ \ q=0,1,2,\ldots,n,\\
&\Omega^{0,\bullet}_m(\widetilde{X},\widetilde{E}):=\set{u\in\Omega^{0,\bullet}(\widetilde{X},\widetilde{E});\, \widetilde{T}u=imu}.
\end{split}
\end{equation} 
For each $m\in\mathbb Z$, we denote by $L^2_m(\widetilde{X},T^{*0,q}\widetilde{X}\otimes \widetilde{E})$ and $L^2_m(\widetilde{X},T^{*0,q}\widetilde{X})$ the completions of $\Omega^{0,q}_m(\widetilde{X},\widetilde{E})$ and $\Omega^{0,q}_m(\widetilde{X})$ with respect to the corresponding pull-back metrics $(\,\cdot\,|\,\cdot\,)_{\widetilde{E}}$ and $(\,\cdot\,|\,\cdot\,)$. Similarly, we denote by $L^2_m(\widetilde{X},T^{*0,\bullet}\widetilde{X}\otimes \widetilde{E})$ and $L^2_m(\widetilde{X},T^{*0,\bullet}\widetilde{X})$ the completions of $\Omega^{0,\bullet}_m(\widetilde{X},\widetilde{E})$ and $\Omega^{0,\bullet}_m(\widetilde{X})$ with respect to the corresponding pull-back metrics $(\,\cdot\,|\,\cdot\,)_{\widetilde{E}}$ and $(\,\cdot\,|\,\cdot\,)$. 

Recall that $U \subset \widetilde{X}$ is called a fundamental domain of the action of $\Gamma$ on $\widetilde{X}$ if the following conditions hold:
\begin{flalign}\label{E:fd}
\begin{split}
& 1. \quad\widetilde X=\cup_{\gamma\in\Gamma}\gamma(\overline{U}),\\
& 2. \quad\gamma_{1}(U)\cap\gamma_{2}(U)=\emptyset \quad \text{for} \quad \gamma_1, \gamma_2\in\Gamma,
\gamma_1\neq\gamma_2, \\
& 3. \quad\overline{U}\setminus U \quad \text{is of measure 0}. \\
\end{split}
\end{flalign}

We can take $U$ to be $S^1$-invariant and with the pull-back $S^1$-action $e^{i\theta}$. We construct such a fundamental domain in the following: From the discussion in the proof of \cite[Theorem 2.11]{CHT}, we can find local trivializations $W_1, \cdots, W_N$ such that $X\, = \, \cup_{j=1}^N W_j$ and each $W_j$ is $S^1$-invariant. For each $j$, let $\widetilde{W}_j \subset \widetilde{X}$ be an $S^1$-invariant open set such that $\pi : \widetilde{W}_j \to W_j$ is a diffeomorphism and a CR map with inverse $\phi_j \, : \, W_j \to \widetilde{W}_j$. Define $U_j \, = \, W_j \backslash (\cup_{i < j}\overline{W}_i \cap W_j)$. Then $U \, := \, \cup_j \phi_j(U_j)$ is the fundamental domain we want.

 It is easy to see that 
\begin{equation}\label{E:3.6.1}
L^2(\widetilde{X}, \widetilde{E}) \simeq L^2\Gamma \otimes L^2(U, \widetilde{E}) \simeq L^2\Gamma \otimes L^2(X, E).
\end{equation}
We then have a unitary action of $\Gamma$ by left translations on $L^2\Gamma$ by $t_\gamma\delta_\eta = \delta_{\gamma\eta}$, where $\left\{ \delta_\eta \, : \, \eta \in \Gamma \right\}$ is the orthonormal basis of $L^2\Gamma$ formed by the delta functions. It induces a unitary action of $\Gamma$ on $L^2(\widetilde{X}, \widetilde{E})$ by $\gamma \mapsto T_\gamma = t_\gamma \otimes \operatorname{Id}$. 

Let us recall the definition of the Von Neumann dimension or $\Gamma$-dimension of a $\Gamma$-module $V \subset L^2(\widetilde{X},T^{*0,q}\widetilde{X}\otimes \widetilde{E})$, see also \cite[Definition 3.6.1]{MM}. We shall denote by $\mathscr{L}(A)$ the space of bounded operators of the Hilbert space $H$. Let $\mathscr{A}_{\Gamma}\subset \mathscr{L}(L^2 \Gamma)$ be the algebra of operators which commute with all left translations and denote the unit element of $\Gamma$ by $e$.
We define Tr$_\Gamma [A]:=\langle A\delta_e,\delta_e \rangle$, $A\in\mathscr{A}_{\Gamma}$. Note that a $\Gamma$-module is a left $\Gamma$-invariant subspace $V\subset L^2 \Gamma$. The orthogonal projection $P_V$ on $V$ is in $\mathscr{A}_{\Gamma}$ for a $\Gamma$-module $V$. Set $\dim_\Gamma V:=\text{Tr}_\Gamma[P_V]$. Now we replace $L^2 \Gamma$ by $L^2(\widetilde{X},T^{*0,q}\widetilde{X}\otimes \widetilde{E})$. Then to any operator 
$A\in\mathscr{L}(L^2(\widetilde{X},T^{*0,q}\widetilde{X}\otimes \widetilde{E}))$, we associate operators
$a_{\gamma\eta}\in \mathscr{L}(L^2(U,T^{*0,q}\widetilde{X}\otimes \widetilde{E}))$ such that $a_{\gamma\eta}(f)$ is the projection of $A(\delta_{\gamma}\otimes f)$ on $\mathbb{C}\delta_\eta \otimes L^2(U,T^{*0,q}\widetilde{X}\otimes \widetilde{E})$. In addition, if $A\in\mathscr{A}_{\Gamma}$ and $A$ is positive, then $a_{\gamma\eta}=a_{e,\gamma^{-1}\eta}$ and 
%\begin{equation}\label{E:3.6.7n9}
\[
\operatorname{Tr}_\Gamma [A] \, := \, \operatorname{Tr}[a_{ee}] \, \geq \, 0,
\]
%\end{equation}
is well-defined. The orthogonal projection $P_V$ on $V\subset L^2(\widetilde{X},T^{*0,q}\widetilde{X}\otimes \widetilde{E})$ is in $\mathscr{A}_{\Gamma}$ for a $\Gamma$-module $V$.
\begin{definition} 
The Von Neumann dimension or $\Gamma$-dimension of a $\Gamma$-module $V$ is defined by 
$$
\dim_\Gamma V:=\operatorname{Tr}_\Gamma[P_V].
$$
\end{definition}

\section{Asymptotic expansion of heat kernels of Kohn Laplacians}\label{S:asympexp}
In this section, we recall the definition of heat kernels.
Then we give a new version of asymptotic expansions of heat kernels of Kohn Laplacians.

\subsection{Asymptotics of heat kernels of Kohn Laplacians on a compact CR manifold}

Since $T\ddbar_b=\ddbar_bT$ and $E$ is a rigid CR vector bundle with a rigid Hermitian metric,
we have 
\[\ddbar_{b,m}:=\ddbar_b:\Omega^{0,\bullet}_m(X,E)\To\Omega^{0,\bullet}_m(X,E),\ \ \forall m\in\mathbb Z.\] 
The $m$-th Fourier component of Kohn-Rossi cohomology is given by
\begin{equation}\label{e-KR}
H^{q}_{b, m}(X,E):=\frac{\Ker \ddbar_{b}: \Omega^{0,q}_m(X,E)\To\Omega^{0,q+1}_m(X,E)  }
{\text{Im} \ddbar_{b}: \Omega^{0,q-1}_m(X,E)\To\Omega^{0,q}_m(X,E) }.
\end{equation}
We also write
\[\ol{\pr}^{*}_b:\Omega^{0,\bullet}(X,E)\To\Omega^{0,\bullet}(X,E)\]
to denote the formal adjoint of $\ddbar_b$ with respect to $(\,\cdot\,|\,\cdot\,)_E$. Since $\langle\,\cdot\,|\,\cdot\,\rangle_E$ and $\langle\,\cdot\,|\,\cdot\,\rangle$ are rigid, we can check that 
\begin{equation}\label{e-gue150517i}
\begin{split}
&T\ddbar^{*}_b=\ddbar^{*}_bT\ \ \mbox{on $\Omega^{0,\bullet}(X,E)$},\\
&\ddbar^{*}_{b,m}:=\ddbar^{*}_b:\Omega^{0,\bullet}_m(X,E)\To\Omega^{0,\bullet}_m(X,E),\ \ \forall m\in\mathbb Z.
\end{split}
\end{equation}
Now, we fix $m\in\mathbb Z$. The $m$-th Fourier component of Kohn Laplacian is given by
\begin{equation}\label{e-gue151113y}
\Box_{b,m}:=(\ddbar_{b,m}+\ddbar^*_{b,m})^2:\Omega^{0,\bullet}_m(X,E)\To\Omega^{0,\bullet}_m(X,E).
\end{equation}
We extend $\Box_{b,m}$ to $L^{2}_m(X,T^{*0,\bullet}X\otimes E)$ by 
\begin{equation}\label{e-gue151113yI}
\Box_{b,m}:{\rm Dom\,}\Box_{b,m}\subset L^{2}_m(X,T^{*0,\bullet}X\otimes E)\To L^{2}_m(X,T^{*0,\bullet}X\otimes E)\,,
\end{equation}
where ${\rm Dom\,}\Box_{b,m}:=\{u\in L^{2}_m(X,T^{*0,\bullet}X\otimes E);\, \Box_{b,m}u\in L^{2}_m(X,T^{*0,\bullet}X\otimes E)\}$, where for any $u\in L^{2}_m(X,T^{*0,\bullet}X\otimes E)$, $\Box_{b,m}u$ is defined in the sense of distributions. 
We recall the following results (see Section 3 in~\cite{CHT}).
\begin{theorem}\label{t-spec}
The Kohn Laplacian $\Box_{b,m}$ is self-adjoint, ${\rm Spec\,}\Box_{b,m}$ is a discrete subset of $[0,\infty[$ and for every $\nu\in{\rm Spec\,}\Box_{b,m}$, $\nu$ is an eigenvalue of $\Box_{b,m}$ with finite multiplicity. 
\end{theorem}
For every $\nu\in{\rm Spec\,}\Box_{b,m}$, let $\set{f^\nu_1,\ldots,f^\nu_{d_{\nu}}}$ be an orthonormal frame for the eigenspace of $\Box_{b,m}$ with eigenvalue $\nu$. The heat kernel $e^{-t\Box_{b,m}}(x,y)$ is given by 
\begin{equation}\label{e-gue151023a}
e^{-t\Box_{b,m}}(x,y)=\sum_{\nu\in{\rm Spec\,}\Box_{b,m}}\sum^{d_{\nu}}_{j=1}e^{-\nu t}f^\nu_j(x)\otimes(f^\nu_j(y))^\dagger,
\end{equation}
where $f^\nu_j(x)\otimes(f^\nu_j(y))^\dagger$ denotes the linear map: 
\[\begin{split}
f^\nu_j(x)\otimes(f^\nu_j(y))^\dagger:T^{*0,\bullet}_yX\otimes E_y&\To T^{*0,\bullet}_xX\otimes E_x,\\
u(y)\in T^{*0,\bullet}_yX\otimes E_y&\To f^\nu_j(x)\langle\,u(y)\,|\,f^\nu_j(y)\,\rangle_E\in T^{*0,\bullet}_xX\otimes E_x.\end{split}\]
Let $e^{-t\Box_{b,m}}:L^2(X,T^{*0,\bullet}X\otimes E)\To L^2_m(X,T^{*0,\bullet}X\otimes E)$ be the continuous operator with distribution kernel $e^{-t\Box_{b,m}}(x,y)$.

We denote by $\dot{\mathcal{R}}$ the Hermitian matrix $\dot{\mathcal{R}} \in \operatorname{End}(T^{1,0}X)$ such that for $V, W \in T^{1,0}X$,
\begin{equation}\label{E:1.5.15}
i d\omega_0(V, \overline{W}) = \langle\,\dot{\mathcal{R}} V\,|\, W \,\rangle.
\end{equation}
Let $\{ \omega_j \}_{j=1}^n$ be a local orthonormal frame of $T^{1,0}X$ with dual frame $\{ \omega^j \}_{j=1}^n$.
Set 
\begin{equation}
\gamma_d = - i\sum^n_{l,j=1} d\omega_0(\omega_j, \overline{\omega}_l) \overline{\omega}^l \wedge \iota_{\overline{\omega}_j},
\end{equation}
where $\iota_{\overline{\omega}_j}$ denotes the interior product of $\overline{\omega}_j$.
Then $\gamma_d \in \operatorname{End}(T^{*0,\bullet}X)$ and $-id\omega_0$ acts as the derivative $\gamma_d$ on $T^{*0,\bullet}X$. 
If we choose $\{ \omega_j \}_{j=1}^n$ to be an orthonormal basis of $T^{1,0}X$ such that
\begin{equation}\label{e-gue160127b}
\dot{\mathcal{R}}(x) = \operatorname{diag} (a_1(x), \cdots, a_n(x)) \in \operatorname{End}(T_x^{1,0}X),
\end{equation}\
then
\begin{equation}\label{E:1.5.19}
\gamma_d(x) = -\sum^n_{j=1} a_j(x) \overline{\omega}^j \wedge \iota_{\overline{\omega}_j}.
\end{equation}
Define ${\rm det\,}\dot{\mathcal{R}}(x):=a_1(x)\cdots a_n(x)$. 

Fix $x, y\in X$. Let $d(x,y)$ denote the standard Riemannian distance of $x$ and $y$ with respect to the given Hermitian metric. Take $\zeta$
\[
o < \zeta < \operatorname{inf} \left\{ \frac{2\pi}{p_k}, \left|\frac{2\pi}{p_r} - \frac{2\pi}{p_{r+1}}\right| , r=1, \cdots, k-1  \right\}.
\]
For $x \in X$, put 
$$
\hat{d}(x,X_{{\rm sing\,}}):=\inf\set{d(x,e^{-i\theta}x);\, \zeta \le \theta \le \frac{2\pi}{p}-\zeta}.
$$

The following result generalizes Theorem 3.1 in \cite{HH16}.

\begin{theorem}\label{T:1.6.1}
	With the above notations and assumptions,
	for every $\epsilon>0$, there are $m_0>0$, $\varepsilon_0>0$ and $C>0$ such that for all $m\geq m_0$, we have
	\begin{equation}\label{E:1.6.4}
	\begin{split}
	&\Big|e^{-\frac{t}{m} \Box_{b,m}}(x,x)-\sum\limits^{p}_{s=1}e^{\frac{2\pi(s-1)}{p}mi} (2\pi)^{-n-1} m^n \frac{\det(\dot{\mathcal{R}}) \exp(t \gamma_d)}{\det(1-\exp(-t\dot{\mathcal{R}}))}(x)  \otimes \operatorname{Id}_{E_x}\Big|\\
	&\leq \epsilon m^n+Cm^nt^{-n}e^{\frac{-\varepsilon_0m \hat{d}(x,X_{{\rm sing\,}})^2}{t}},\ \ \forall (t,x)\in \Real_{+}\times X_{{\rm reg\,}}.
	\end{split}
	\end{equation}
\end{theorem}

\begin{proof}
We use the notations from Section 3 in \cite{HH16}. Recall that $\Gamma_{m}$ is defined in \cite[(3.31)]{HH16} (see also \eqref{e-gue150626fIII2}). For $x\in X_{\rm reg}$, we have
\begin{equation}  \label{E:1.7}
\begin{split}
\Gamma_m(t,x,x) 
&=\frac{1}{2\pi}\sum^N_{j=1}\int^{2\pi}_{0}H_{j,m}(t,x,e^{iu}\circ
x)e^{imu}du\\
&=\frac{1}{2\pi}\sum\limits^{p}_{s=1}e^{\frac{2\pi(s-1)}{p}mi}\sum^N_{j=1}
\int^{\frac{2\pi}{p}}_{0}H_{j,m}(t,x,e^{iu}\circ x)e^{imu}du\\
&=\frac{1}{2\pi}\sum\limits^{p}_{s=1}e^{\frac{2\pi(s-1)}{p}mi}\sum^N_{j=1}
\int_{u\in[\zeta,\frac{2\pi}{p}-\zeta]}H_{j,m}(t,x,e^{iu}\circ
x)e^{imu}du\\
&\quad+\frac{1}{2\pi}\sum\limits^{p}_{s=1}e^{\frac{2\pi(s-1)}{p}mi}\sum^N_{j=1}
\int^{\zeta}_{-\zeta}H_{j,m}(t,x,e^{iu}\circ x)e^{imu}du,
\end{split}
\end{equation}
where $H_{j,m}$ is defined in \cite[(3.30)]{HH16} (see also \eqref{e-gue150626fIII2}).
From \cite[(3.29), (3.34)]{HH16} and \cite[(6.4)]{CHT}, there are $\varepsilon_0 > 0$ and $C_0$ independent of $j, x, m, t$ such that, for all $t\in \Real_{+}$ and for all $m \in \mathbb{N},$ we have
\begin{equation}\label{E:3.38HH16}
\left| \frac{1}{2\pi}\int_{u\in[\zeta,\frac{2\pi}{p}-\zeta]}H_{j,m}(t,x,e^{iu}\circ
x)e^{imu}du \right| \le C_0m^nt^{-n}e^{\frac{-\varepsilon_0m \hat{d}(x,X_{{\rm sing\,}})^2}{t}}.
\end{equation}

Then the proof is completed by applying \cite[(3.32), (3.39)]{HH16} and \eqref{E:3.38HH16}.

\begin{remark}\label{r-01}
It is easy to check that
\begin{equation}\label{e-23}
\sum\limits^{p}_{s=1}e^{\frac{2\pi(s-1)}{p}mi}
=\begin{cases}
p& p\mid m\\
0& p\nmid m.
\end{cases}
\end{equation}
\end{remark}
\end{proof}

%%%%%%%%%%%%%%%%%%%%%%%%%%%%%%%%%%%%%%%%%%%%%%%%%%%%%%%%%%%%%%%%%%%%%%%%%%

\subsection{BRT trivializations}

To prove Theorem~\ref{T:1.6.1}, we need some preparations. We first need the following result due to Baouendi-Rothschild-Treves~\cite{BRT85}.

\begin{theorem}\label{t-gue150514}
For every point $x_0\in X$, we can find local coordinates $x=(x_1,\cdots,x_{2n+1})=(z,\theta)=(z_1,\cdots,z_{n},\theta), z_j=x_{2j-1}+ix_{2j},j=1,\cdots,n, x_{2n+1}=\theta$, defined in some small neighborhood $D=\{(z, \theta): \abs{z}<\delta, -\varepsilon_0<\theta<\varepsilon_0\}$ of $x_0$, $\delta>0$, $0<\varepsilon_0<\pi$, such that $(z(x_0),\theta(x_0))=(0,0)$ and 
\begin{equation}\label{e-can}
\begin{split}
&T=\frac{\partial}{\partial\theta}\\
&Z_j=\frac{\partial}{\partial z_j}+i\frac{\partial\varphi}{\partial z_j}(z)\frac{\partial}{\partial\theta},j=1,\cdots,n
\end{split}
\end{equation}
where $Z_j(x), j=1,\cdots, n$, form a basis of $T_x^{1,0}X$, for each $x\in D$, and $\varphi(z)\in C^\infty(D,\mathbb R)$ is independent of $\theta$. We call $(D,(z,\theta),\varphi)$ BRT trivialization.
\end{theorem}

By using BRT trivialization, we get another way to define $Tu, \forall u\in\Omega^{0,q}(X)$. Let $(D,(z,\theta),\varphi)$ be a BRT trivialization. It is clear that
$$\{d\overline{z_{j_1}}\wedge\cdots\wedge d\overline{z_{j_q}}, 1\leq j_1<\cdots<j_q\leq n\}$$
is a basis for $T^{\ast0,q}_xX$, for every $x\in D$. Let $u\in\Omega^{0,q}(X)$. On $D$, we write
\begin{equation}\label{e-gue150524fb}
u=\sum\limits_{1\leq j_1<\cdots<j_q\leq n}u_{j_1\cdots j_q}d\overline{z_{j_1}}\wedge\cdots\wedge d\overline{z_{j_q}}.
\end{equation}
Then, on $D$, we can check that
\begin{equation}\label{lI}
Tu=\sum\limits_{1\leq j_1<\cdots<j_q\leq n}(Tu_{j_1\cdots j_q})d\overline{z_{j_1}}\wedge\cdots\wedge d\overline{z_{j_q}}
\end{equation}
and $Tu$ is independent of the choice of BRT trivializations. Note that, on BRT trivialization $(D,(z,\theta),\varphi)$, we have 
\begin{equation}\label{e-gue150514f}
\ddbar_b=\sum^n_{j=1}d\ol z_j\wedge(\frac{\partial}{\partial\ol z_j}-i\frac{\partial\varphi}{\partial\ol z_j}(z)\frac{\partial}{\partial\theta}).
\end{equation}

%%%%%%%%%%%%%%%%%%%%%%%%%%%%%%%%%%%%%%%%%%%%%%%%%%%%%%%%%%%%%%%%%%%%%%%%%%%

\subsection{Local heat kernels on BRT trivializations}

Until further notice, we fix $m\in\mathbb Z$. 
Let $B:=(D,(z,\theta),\varphi)$ be a BRT trivialization. We may assume that $D=U\times]-\varepsilon,\varepsilon[$, where $\varepsilon>0$ and $U$ is an open set of $\Complex^n$. Since $E$ is rigid, we can consider $E$ as a holomorphic vector bundle over $U$. We may assume that $E$ is trivial on $U$. Consider a trivial line bundle $L\To U$ with non-trivial Hermitian fiber metric $\abs{1}^2_{h^L}=e^{-2\varphi}$. Let $(L^m,h^{L^m})\To U$ be the $m$-th power of $(L,h^L)$.  Let $\Omega^{0,q}(U,E\otimes L^m)$ and $\Omega^{0,q}(U,E)$ be the spaces of $(0,q)$ forms on $U$ with values in $E\otimes L^m$ and $E$, respectively, $q=0,1,2,\ldots,n$. Put 
\[
\begin{split}
&\Omega^{0,\bullet}(U,E\otimes L^m):=\oplus_{j\in\set{0,1,\ldots,n}}\Omega^{0,j}(U,E\otimes L^m),\\
&\Omega^{0,\bullet}(U,E):=\oplus_{j\in\set{0,1,\ldots,n}}\Omega^{0,j}(U,E).
\end{split}\]
Since $L$ is trivial, from now on, we identify $\Omega^{0,\bullet}(U,E)$ with $\Omega^{0,\bullet}(U,E\otimes L^m)$.
Since the Hermitian fiber metric $\langle\,\cdot\,|\,\cdot\,\rangle_E$ is rigid, we can consider $\langle\,\cdot\,|\,\cdot\,\rangle_E$ as a Hermitian fiber metric on the holomorphic vector bundle $E$ over $U$. 
Let $\langle\,\cdot\,,\,\cdot\,\rangle$ be the Hermitian metric on $\Complex TU$ given by 
\[\langle\,\frac{\pr}{\pr z_j}\,,\,\frac{\pr}{\pr z_k}\,\rangle=\langle\,\frac{\pr}{\pr z_j}+i\frac{\pr\varphi}{\pr z_j}(z)\frac{\pr}{\pr\theta}\,|\,\frac{\pr}{\pr z_k}+i\frac{\pr\varphi}{\pr z_k}(z)\frac{\pr}{\pr\theta}\,\rangle,\ \ j,k=1,2,\ldots,n.\]
$\langle\,\cdot\,,\,\cdot\,\rangle$ induces a Hermitian metric on $T^{*0,\bullet}U:=\oplus_{j=0}^nT^{*0,j}U$, where $T^{*0,j}U$ is the bundle of $(0,j)$ forms on $U$, $j=0,1,\ldots,n$. We shall also denote this induced Hermitian metric on $T^{*0,\bullet}U$ by $\langle\,\cdot\,,\,\cdot\,\rangle$. The Hermitian metrics on $T^{*0,\bullet}U$ and $E$ induce a
Hermitian metric on $T^{*0,\bullet}U\otimes E$. We shall also denote this induced metric by $\langle\,\cdot\,|\,\cdot\,\rangle_{E}$.
Let $(\,\cdot\,,\,\cdot\,)$ be the $L^2$ inner product on $\Omega^{0,\bullet}(U,E)$ induced by $\langle\,\cdot\,,\,\cdot\,\rangle$, $\langle\,\cdot\,|\,\cdot\,\rangle_E$. Similarly, let $(\,\cdot\,,\,\cdot\,)_m$ be the $L^2$ inner product on $\Omega^{0,\bullet}(U,E\otimes L^m)$ induced by $\langle\,\cdot\,,\,\cdot\,\rangle$, $\langle\,\cdot\,|\,\cdot\,\rangle_E$ and $h^{L^m}$. 

The curvature of $L$ induced by $h^L$ is given by $R^L:=2\pr\ddbar\varphi$. Let $\dot{R}^L\in \operatorname{End}(T^{1,0}U)$ be the Hermitian matrix given by 
\[R^L(W,\ol Y)=\langle\,\dot{R}^LW\,,\, Y\,\rangle,\ \ W, Y\in T^{1,0}U.\]
Let $\{w_j \}_{j=1}^n$ be a local orthonormal frame of $T^{1,0}U$ with dual frame $\{ w^j \}_{j=1}^n$.
Set 
\begin{equation}\label{e-gue160128a}
\omega_d = - \sum_{l,j}R^L(w_j, \overline{w}_l) \overline{w}^l \wedge \iota_{\overline{w}_j},
\end{equation}
where $\iota_{\overline{w}_j}$ denotes the interior product of $\overline{w}_j$. 

%Let $\mathcal{A}_\ell \in C^\infty(U, \operatorname{End}(T^{*0,\bullet}U))$, $\ell=-n,-n+1,\ldots$, be such that, for every $k\in\mathbb Z$, 
%\begin{equation}\label{e-gue160128aI}
%(2\pi)^{-n} \frac{\det(\dot{R}^L) \exp(t \omega_d)}{\det(1-\exp(-t\dot{R}^L))}(z) = \sum_{\ell=-n}^k\mathcal{A}_\ell(z)t^\ell+O(t^{k+1})
%\end{equation}
%in $C^0(U, \operatorname{End}(T^{*0,\bullet}U))$ locally uniformly on $U$. Let 
%\[\begin{split}
%\pi:D&\To U,\\
%(z,\theta)&\To z
%\end{split}\]
%be the natural projection.
%It is straightforward to check that 
%\begin{equation}\label{e-gue160129I}
%(2\pi)^{-n} \frac{\det(\dot{R}^L)\exp(t\omega_d)}{\det(1-\exp(-t\dot{R}^L))}(\pi(x)) \otimes \operatorname{Id}_{E_{\pi(x)}} = (2\pi)^{-n} \frac{\det(\dot{\mathcal{R}}) \exp(t \gamma_d)}{\det(1-\exp(-t\dot{\mathcal{R}}))}(x)  \otimes \operatorname{Id}_{E_{x}},\ \ \forall x\in D.
%\end{equation}
% From \eqref{E:5.5.34}, \eqref{e-gue160128aI} and \eqref{e-gue160129I}, it is easy to see that 
%\begin{equation}\label{e-gue160131}
%\mathcal{A}_\ell(\pi(x))=(2\pi)A_\ell(x),\ \ \forall x\in D,\ \ \ell=-n,-n+1,\ldots. 
%\end{equation}

Let
\[\ddbar:\Omega^{0,\bullet}(U,E\otimes L^m)\To\Omega^{0,\bullet}(U,E\otimes L^m)\]
be the Cauchy-Riemann operator and let
\[\ol{\pr}^{*,m}:\Omega^{0,\bullet}(U,E\otimes L^m)\To\Omega^{0,\bullet}(U,E\otimes L^m)\] 
be the formal adjoint of $\ddbar$ with respect to $(\,\cdot\,,\,\cdot\,)_m$. Put 
\begin{equation}\label{e-gue150606II}
\Box_{B,m}:=(\ddbar+\ol{\pr}^{*,m})^2: \Omega^{0,\bullet}(U,E\otimes L^m)\To\Omega^{0,\bullet}(U,E\otimes L^m).
\end{equation}
We need the following result (see Lemma 5.1 in~\cite{CHT}) 

\begin{lemma}\label{l-gue150606}
Let $u\in\Omega^{0,\bullet}_m(X,E)$. On $D$, we write $u(z,\theta)=e^{im\theta}\Td u(z)$, $\Td u(z)\in\Omega^{0,\bullet}(U,E)$. Then,
\begin{equation}\label{e-gue150606III}
e^{-m\varphi}\Box_{B,m}(e^{m\varphi}\Td u)=e^{-im\theta}\Box_{b,m}(u).
\end{equation}
\end{lemma}

Let $z, w\in U$ and let $T(z,w)\in (T^{*0,\bullet}_wU\otimes E_w)\boxtimes(T^{*0,\bullet}_zU\otimes E_z)$. We write $\abs{T(z,w)}$ to denote the standard pointwise matrix norm of $T(z,w)$ induced by $\langle\,\cdot\,|\,\cdot\,\rangle_E$. Let $\Omega^{0,\bullet}_0(U,E)$ be the subspace of $\Omega^{0,\bullet}(U,E)$ whose elements have compact support in $U$. Let $dv_U$ be the volume form on $U$ induced by $\langle\,\cdot\,,\,\cdot\,\rangle$. Assume $T(z,w)\in C^\infty(U\times U,(T^{*0,\bullet}_wU\otimes E_w)\boxtimes(T^{*0,\bullet}_zU\otimes E_z))$. Let $u\in\Omega^{0,\bullet}_0(U,E)$. We define the integral $\int T(z,w)u(w)dv_U(w)$ in the standard way. Let $G(t,z,w)\in C^\infty(\Real_+\times U\times U,(T^{*0,\bullet}_wU\otimes E_w)\boxtimes(T^{*0,\bullet}_zU\otimes E_z))$. We write $G(t)$ to denote the continuous operator
\[\begin{split}
G(t):\Omega^{0,\bullet}_0(U,E)&\To\Omega^{0,\bullet}(U,E),\\
u&\To\int G(t,z,w)u(w)dv_U(w)\end{split}\]
and we write
$G'(t)$ to denote the continuous operator
\[\begin{split}
G'(t):\Omega^{0,\bullet}_0(U,E)&\To\Omega^{0,\bullet}(U,E),\\
u&\To\int \frac{\pr G(t,z,w)}{\pr t}u(w)dv_U(w).\end{split}\]

We consider the heat operator of $\Box_{B,m}$. By using the standard Dirichlet heat kernel construction (see~\cite{G8}) and the proofs of Theorem 1.6.1 and Theorem 5.5.9 in~\cite{MM}, we deduce the following 

\begin{theorem}\label{t-gue150607}
There is $A_{B,m}(t,z,w)\in C^\infty(\Real_+\times U\times U,(T^{*0,\bullet}_wU\otimes E_w)\boxtimes(T^{*0,\bullet}_zU\otimes E_z))$ such that 
\begin{equation}\label{e-gue150607ab}
\begin{split}
&\mbox{$\lim_{t\To0+}A_{B,m}(t)=I$ in $D'(U,T^{*0,\bullet}U\otimes E)$},\\
&A'_{B,m}(t)u+\frac{1}{m}A_{B,m}(t)(\Box_{B,m}u)=0,\ \ \forall u\in\Omega^{0,\bullet}_0(U,E),\ \ \forall t>0,
\end{split}\end{equation}
and $A_{B,m}(t,z,w)$ satisfies the following:

(I)  For every compact set $K\Subset U$, $\alpha_1, \alpha_2, \beta_1, \beta_2\in\mathbb N^n_0$,  there are constants $C_{\alpha_1,\alpha_2,\beta_1,\beta_2,K}>0$ and $\varepsilon_0>0$ independent of $t$ and $m$ such that
\begin{equation}\label{e-gue160128w}
\begin{split}
&\abs{\pr^{\alpha_1}_z\pr^{\alpha_2}_{\ol z}\pr^{\beta_1}_w\pr^{\beta_2}_{\ol w}\Bigr(A_{B,m}(t,z,w)e^{m(\varphi(w)-\varphi(z))}\Bigr)}\\
&\leq C_{\alpha_1,\alpha_2,\beta_1,\beta_2,K}(\frac{m}{t})^{n+\abs{\alpha_1}+\abs{\alpha_2}+\abs{\beta_1}+\abs{\beta_2}}e^{-m\varepsilon_0\frac{\abs{z-w}^2}{t}},\ \  \forall (t,z,w)\in \Real_+\times K\times K.
\end{split}
\end{equation}

(II) $A_{B,m}(t,z,z)$ admits an asymptotic expansion: 
\begin{equation}\label{E:1.6.4b}
A_{B,m}(t,z,z)=(2\pi)^{-n} m^n \frac{\det(\dot{R}^L)\exp(t\omega_d)}{\det(1-\exp(-t\dot{R}^L))}(z)  \otimes \operatorname{Id}_{E_z} +o(m^n)
\end{equation}
in $C^\ell(U,\operatorname{End}(T^{*0,\bullet}U) \otimes E)$ locally uniformly on $\Real_+\times U$, for every $\ell\in\mathbb N$. Here we use the convention that if an eigenvalue $a_j(z)$ of $\dot{R}^L(z)$ is zero, then its contribution for $\frac{\det(\dot{R}^L)}{\det(1-\exp(-t\dot{R}^L))}(z)$ is $\frac{1}{t}$. 

\end{theorem}

%%%%%%%%%%%%%%%%%%%%%%%%%%%%%%%%%%%%%%%%%%%%%%%%%%%%%%%%%%%%%%%%%%%%%%%%%%

\subsection{$L^2$ Kohn-Rossi cohomology on a covering manifold}

Let 
$$
\widetilde{\Box}_{b} \, : \, \operatorname{Dom} \widetilde{\Box}_{b} \subset L^2(\widetilde{X}, T^{*0,\bullet}\widetilde{X}) \to L^2(\widetilde{X}, T^{*0,\bullet}\widetilde{X})
$$
be the Gaffney extension of the pull-back Kohn Laplacian on $\widetilde{X}$. By a result of Gaffney, $\widetilde{\Box}_{b}$ is a positive self-adjoint operator (see Proposition 3.1.2 in Ma-Marinescu \cite{MM}). That is, $\widetilde{\Box}_{b}$ is self-adjoint and the spectrum of $\widetilde{\Box}_{b}$ is contained in $\overline{\mathbb{R}}_+.$ Now, we fix $m\in\mathbb Z$. As in \eqref{e-gue151113y}, we introduce the $m$-th Fourier component of the Kohn Laplacian $\widetilde{\Box}_{b,m}$ on $\Omega^{0,\bullet}_m(\widetilde{X},\widetilde{E})$. We can easily see that $\widetilde{\Box}_{b,m}$ is also self-adjoint.
By the second isomorphism of \eqref{E:3.6.1}, we can see that, for any $\gamma \in \Gamma$, 
\begin{equation}\label{E:3.6.13}
T_\gamma (\operatorname{Dom}(\widetilde{\Box}_{b,m})) \subset \operatorname{Dom}(\widetilde{\Box}_{b,m}), \quad 
T_\gamma \widetilde{\Box}_{b,m}\, = \, \widetilde{\Box}_{b,m} T_\gamma \quad  \text{on}  \quad    \operatorname{Dom}(\widetilde{\Box}_{b,m}).
\end{equation}

Consider the spectral resolution $E^q_\lambda(\widetilde{\Box}_{b,m})$ of $\widetilde{\Box}_{b,m}$ acting on $L^2_m(\widetilde{X},T^{*0,q}\widetilde{X}\otimes \widetilde{E})$. (See \cite[Appendix C.2]{MM}). The proof of the following lemma is similar to Lemma 3.6.3 in Ma-Marinescu \cite{MM}.
\begin{lemma}
For any $q = 0, 1, \cdots, n$ and $\lambda \in \mathbb{R}$, then $E^q_\lambda(\widetilde{\Box}_{b,m})$ commutes with $\Gamma$, its Schwartz kernel is smooth and 
$$
\dim_\Gamma E^q_\lambda(\widetilde{\Box}_{b,m}) < +\infty.
$$
\end{lemma}
\begin{proof}
By \eqref{E:3.6.1} and \eqref{E:3.6.13}, we can see that, for any $\lambda \in \mathbb{R}$, $E^q_\lambda(\widetilde{\Box}_{b,m})$ commutes with $\Gamma$. 
We claim that $\widetilde{\Box}_{b,m}-\widetilde{T}^2 \equiv \Delta$ is a second order elliptic operator,
so is $\Delta -m^2$. Its principal symbol is locally written as
\begin{equation*}
\sigma_{\Delta}(\widetilde x,\xi)=\sigma_{\widetilde{\Box}_{b,m}}(\widetilde x,\xi)-\sigma_{\widetilde{T}^2}(\widetilde x,\xi)=\sum_{j=1}^{n}|\sigma_{L_{j}}(\widetilde x,\xi)|^2-\sigma_{\widetilde{T}}(\widetilde x,\xi)^2,
\end{equation*}
where $\xi=(\xi_{1},...,\xi_{2n},\xi_{2n+1})$ and $\{L_j\}$ is an orthonormal basis of $T^{0,1}_x \widetilde X$.
It is well-known that the characteristic manifold of $\widetilde{\Box}_{b}$ is 
\begin{equation*}
\Sigma=\{(\widetilde x, c\widetilde \omega_{0}(\widetilde x)) \in T^* \widetilde X :c\neq 0\}.
\end{equation*}
It means that $\sigma_{\widetilde{\Box}_{b,m}}(\widetilde x,\xi)>0$ if and only if $(\xi_{1},...,\xi_{2n})\neq 0$.
Meanwhile, in a local BRT coordinate \cite{BRT85}, 
we have $\widetilde T=\frac{\partial}{\partial\theta}$, then $\sigma_{\widetilde T}=i\xi_{2n+1}$. That is, $\sigma_{\widetilde{T}^2}=-\xi_{2n+1}^2$. Then the claim is proved. By the spectral theorem, cf. \cite[Theorem C.2.1]{MM}, we have $\operatorname{Im} (E_\lambda(\Delta -m^2)) \subset \operatorname{Dom} ((\Delta -m^2)^k)$ for $k \in \mathbb{N}$. Using the uniform Sobolev spaces \cite[pp. 511-512]{Sh}, it is easy to see that $\operatorname{Im} (E_\lambda(\Delta -m^2)) \subset \Omega^{\bullet}(\widetilde{X}, \widetilde{E})$, so that $E_\lambda(\Delta -m^2) \, : \, L^2(\widetilde{X},T^{*0,\bullet}\widetilde{X}\otimes \widetilde{E}) \to \Omega^{\bullet}(\widetilde{X}, \widetilde{E})$ is linear continuous. Hence,
 $\operatorname{Im} E_\lambda(\widetilde{\Box}_{b,m}) = \operatorname{Im} (E_\lambda(\Delta -m^2)) \cap L^2_m(\widetilde{X},T^{*0,\bullet}\widetilde{X}\otimes \widetilde{E}) \subset \Omega^{\bullet}(\widetilde{X}, \widetilde{E}) \cap L^2_m(\widetilde{X},T^{*0,\bullet}\widetilde{X}\otimes \widetilde{E})= \Omega^{\bullet}_m(\widetilde{X}, \widetilde{E}) $ and $E_\lambda(\widetilde{\Box}_{b,m}) \, : \, L^2_m(\widetilde{X},T^{*0,\bullet}\widetilde{X}\otimes \widetilde{E}) \to \Omega^{\bullet}_m(\widetilde{X}, \widetilde{E})$ is also linear continuous. By Schwartz kernel theorem, the kernel $E_\lambda(\widetilde{\Box}_{b,m})(\widetilde{x}, \widetilde{x})$ of $E_\lambda(\widetilde{\Box}_{b,m})$ with respect to $dv_{\widetilde{X}}(\widetilde{x})$ is smooth. By \cite[(3.6.12)]{MM},
 \[
 \dim_\Gamma E_\lambda(\widetilde{\Box}_{b,m})\, = \, \int_U \operatorname{Tr} [E_\lambda(\widetilde{\Box}_{b,m})(\widetilde{x}, \widetilde{x})]dv_{\widetilde{X}}(\widetilde{x}) < +\infty.
 \]
\end{proof}

\begin{definition} \label{D815}
\begin{enumerate}
\item The $m$-th Fourier component of the space of harmonic forms $\mathcal{H}^{\bullet}(\widetilde{X}, \widetilde{E})$ is defined by
\[
\mathcal{H}^{\bullet}_{b, m}(\widetilde{X}, \widetilde{E}) \, := \, \operatorname{Ker} (\widetilde{\Box}_{b,m}) \, = \, \left\{s \in \operatorname{Dom}\widetilde{\Box}_{b,m} : \widetilde{\Box}_{b,m}s=0  \right\}.
\]
\item The $m$-th Fourier component of the $q$-th reduced $L^2$ Kohn-Rossi cohomology is given by
\begin{equation}\label{e-L2KR}
\overline{H}^{q}_{b, (2), m}(\widetilde{X},\widetilde{E}) \, := \, \frac{\Ker \ddbar_{b} \cap L^2_m ({\widetilde{X}, T^{*0,q}\widetilde{X} \otimes \widetilde{E}) }}
{\big[ \operatorname{Im} \ddbar_{b} \cap  L^2_m ( {\widetilde{X}, T^{*0,q}\widetilde{X} \otimes \widetilde{E}) } \big]},
\end{equation}
where $[V]$ denotes the closure of the space $V$.
\end{enumerate}
\end{definition}
We can easily obtain the following weak Hodge decomposition
\begin{equation}\label{E:weakHodge}
L^2_m ( \widetilde{X}, T^{*0,\bullet}\widetilde{X} \otimes \widetilde{E}) \, = \,  \mathcal{H}^{\bullet}(\widetilde{X}, \widetilde{E}) \oplus 
[\operatorname{Im}(\overline{\partial}_{b,m})] \oplus [ \operatorname{Im}(\overline{\partial}_{b,m}^*)]
\end{equation}
By \eqref{E:weakHodge}, we the the isomorphism
\begin{equation}\label{E:canonisom}
\overline{H}^{\bullet}_{b, (2), m}(\widetilde{X},\widetilde{E}) \cong \mathcal{H}^{\bullet}_b(\widetilde{X}, \widetilde{E}).
\end{equation}

\subsection{Asymptotics of heat kernels of Kohn Laplacians on a covering manifold}

Assume that $X=D_1\bigcup D_2\bigcup\cdots\bigcup D_N$, where $B_j:=(D_j,(z,\theta),\varphi_j)$ is a BRT trivialization, for each $j$. We may assume that, for each $j$, $D_j=U_j\times]-2\delta_j,2\Td\delta_j[\subset\Complex^n\times\Real$, $\delta_j>0$, $\Td\delta_j>0$, $U_j=\set{z\in\Complex^n;\, \abs{z}<l_j}$. For each $j$, put $\hat D_j=\hat U_j\times]-\frac{\delta_j}{2},\frac{\Td\delta_j}{2}[$, where $\hat U_j=\set{z\in\Complex^n;\, \abs{z}<\frac{l_j}{2}}$. We may suppose that $X=\hat D_1\bigcup\hat D_2\bigcup\cdots\bigcup\hat D_N$. 

Let $\left\{ \psi_j \right\}$ be a partition of unity subordinate to $\left\{ \hat D_j \right\}$. Then $\left\{ \widetilde{\psi}_{\gamma, j} := \psi_i \circ \pi \right\}$ is a partition of unity subordinate to $\left\{ \widetilde{D}_{\gamma, j} \right\}$, where $\pi^{-1}(\hat{D}_j) = \cup_{\gamma \in \Gamma} \widetilde{D}_{\gamma, j}$ and $\widetilde{D}_{\gamma_1, j}$ and $\widetilde{D}_{\gamma_2, j}$ are disjoint for $\gamma_1 \not= \gamma_2$. For each $\gamma \in \Gamma$ and each $j$, we have $\widetilde{D}_{\gamma, j}=\widetilde{U}_{\gamma, j}\times]-\frac{\delta_{\gamma, j}}{2},\frac{\Td\delta_{\gamma, j}}{2}[$, where $\widetilde{U}_{\gamma, j}=\set{z\in\Complex^n;\, \abs{z}<\frac{l_{\gamma, j}}{2}}$. Then $\widetilde{X} = \bigcup_{\gamma \in \Gamma}\bigcup_{j=1}^N \widetilde{D}_{\gamma, j}$.

%Let $\chi_j\in C^\infty_0(\hat D_j)$, $j=1,2,\ldots,N$, with $\sum^N_{j=1}\chi_j=1$ on $X$. 

Fix $\gamma \in \Gamma$ and $j=1,2,\ldots,N$. Put 
\[
K_{\gamma, j}=\set{z\in\widetilde{U}_{\gamma, j};\, \mbox{there is a $\theta\in]-\frac{\delta_{\gamma, j}}{2},\frac{\Td\delta_{\gamma, j}}{2}[$ such that $\widetilde{\psi}_{\gamma, j}(z,\theta)\neq0$}}.
\]
Let $\tau_{\gamma, j}(z)\in C^\infty_0(\widetilde{U}_{\gamma, j})$ with $\tau_{\gamma, j}\equiv1$ on some neighborhood $W_{\gamma, j}$ of $K_{\gamma, j}$. Let $\sigma_{\gamma, j}\in C^\infty_0(]-\frac{\delta_{\gamma, j}}{2},\frac{\Td\delta_{\gamma, j}}{2}[)$ with $\int\sigma_{\gamma, j}(\theta)d\theta=1$. Let $\widetilde{A}_{B_{\gamma, j},m}(t,z,w)\in C^\infty(\Real_+\times \widetilde{U}_{\gamma, j}\times \widetilde{U}_{\gamma, j},(T^{*0,\bullet}_w\widetilde{U}_{\gamma, j}\otimes \widetilde{E}_w)\boxtimes(T^{*0,\bullet}_z\widetilde{U}_{\gamma, j}\otimes \widetilde{E}_z))$ be as in Theorem~\ref{t-gue150607}. 

Put 
\begin{equation}\label{e-gue150627f1}
\widetilde{H}_{\gamma, j,m}(t,\widetilde{x},\widetilde{y})=\widetilde{\psi}_{\gamma, j}(\widetilde{x})e^{-m\varphi_j(z)+im\theta}\widetilde{A}_{B_{\gamma, j},m}(t,z,w)e^{m\varphi_{\gamma, j}(w)-im\eta}\tau_{\gamma, j}(w)\sigma_{\gamma, j}(\eta),
\end{equation}
where $\widetilde{x}=(z,\theta)$, $\widetilde{y}=(w,\eta)\in\Complex^{n}\times\Real$. Let 
\begin{equation}\label{e-gue150626fIII2}
\begin{split}
\widetilde{\Gamma}_m(t,\widetilde{x},\widetilde{y}):=\frac{1}{2\pi}\sum_{\gamma \in \Gamma}\sum^N_{j=1}\int^\pi_{-\pi}\widetilde{H}_{\gamma, j,m}(t,\widetilde{x},e^{iu}\circ \widetilde{y})e^{imu}du.
\end{split}
\end{equation}

Note that when $\Gamma = \left\{ e \right\}$, $\widetilde{\Gamma}_m(t,\widetilde{x},\widetilde{y})=\Gamma_m(t, \pi(\widetilde{x}) ,\pi(\widetilde{y}))$ is defined in \cite[(3.31)]{HH16}.

From Lemma~\ref{l-gue150606}, off-diagonal estimates of $\widetilde{A}_{B_j,m}(t,\widetilde{x},\widetilde{y})$ (see \eqref{e-gue160128w}), we can repeat the proof of Theorem 5.14 in~\cite{CHT} with minor change and deduce that

\begin{theorem}\label{t-gue150630I}
For every $\ell\in\mathbb N$, $\ell\geq2$, and every $M>0$, there are $\epsilon_0>0$ and $m_0>0$ independent of $t$ and $m$ such that for every $m\geq m_0$, we have
\begin{equation}\label{e-gue150630g}
\norm{  e^{-\frac{t}{m}\widetilde{\Box}_{b,m}}(\widetilde{x},\widetilde{y})-\widetilde{\Gamma}_m(t,\widetilde{x},\widetilde{y}) }_{C^l(\widetilde{X} \times \widetilde{X})} \leq e^{-\frac{m}{t}\epsilon_0},\ \ \forall t\in(0,M).
\end{equation}
\end{theorem}

From Theorem 3.6.4 in \cite{MM}, we have
\begin{proposition}\label{e-que1808240935}
For any $t_0 > 0, \varepsilon > 0$ and any $\gamma \in \Gamma, j=1, 2, \cdots, N$, there exists $C>0$ such that for any $z \in \widetilde{U}_{\gamma, j}, m \in \mathbb{N}, t > t_0$,
$$
\norm{ \widetilde{A}_{B_{\gamma, j},m}(t,z,z) - A_{B_j, m}(t, \pi(z), \pi(z)) }_{C^l(\widetilde{U}_{\gamma, j} \times \widetilde{U}_{\gamma, j})} \le C \exp \left( -\frac{ m}{32t}\varepsilon \right).
$$
\end{proposition}

From \eqref{E:1.7} (see (3.31) in \cite{HH16}), \eqref{e-gue150627f1}, \eqref{e-gue150626fIII2}, Proposition~\ref{e-que1808240935} and the fact that $\widetilde{\psi}_{\gamma, j} = \psi_j \circ \pi$, we can easily deduce that
\begin{lemma}\label{e-que201808241012}
With the above notations and assumptions as in Theorem~\ref{t-gue150630I}, we have
\[
\norm{ \widetilde{\Gamma}_m(t,\widetilde{x},\widetilde{x}) - \Gamma_m (t,\pi(\widetilde{x}),\pi(\widetilde{x})) }_{C^l(\widetilde{X} \times \widetilde{X})} \le C \exp \left( -\frac{m}{t}\epsilon_0 \right).
\]
\end{lemma}

From Theorem~\ref{t-gue150630I}, Lemma~\ref{e-que201808241012} and Theorem 3.5 of \cite{HH16}, we have
\begin{theorem}\label{t-que201808241022}
For every $\ell\in\mathbb N$, $\ell\geq2$, and every $M>0$, there are $\epsilon_0 >0$ and $m_0 > 0$ independent of $t$ and $m$ such that for any $\widetilde{x} \in \widetilde{X}$ and $m \ge m_0$, we have
\[
\norm { e^{-\frac{t}{m}\widetilde{\Box}_{b,m}}(\widetilde{x},\widetilde{x}) - e^{-\frac{t}{m}\Box_{b,m}}(\pi(\widetilde{x}),\pi(\widetilde{x})) }_{C^l(\widetilde{X} \times \widetilde{X})} \le C \exp \left( -\frac{ m}{t}\epsilon_0 \right), \quad \forall t \in (0, M).
 \]
\end{theorem}

%%%%%%%%%%%%%%%%%%%%%%%%%%%%%%%%%%%%%%%%%%%%%%%%%%%%%%%%%%%%%%%%%%%%%%%%%%%

By Theorem \ref{T:1.6.1} and Theorem \ref{t-que201808241022}, we have

\begin{theorem}\label{C:3.6.5}
	With the above notations and assumptions,
	for every $\epsilon>0$, there are $m_0>0$, $\varepsilon_0>0$ and $C>0$ such that for all $m\geq m_0$, we have
	\begin{equation}\label{E:3.6.19}
	\begin{split}
	&\Big|e^{-\frac{t}{m} \widetilde{\Box}_{b,m}}(\widetilde{x},\widetilde{x})-\sum\limits^{p}_{s=1}e^{\frac{2\pi(s-1)}{p}mi} (2\pi)^{-n-1} m^n \frac{\det(\dot{\mathcal{R}}) \exp(t \gamma_d)}{\det(1-\exp(-t\dot{\mathcal{R}}))}(\pi(\widetilde{x}))  \otimes \operatorname{Id}_{E_{\pi(\widetilde{x})}}\Big|\\
	&\leq \epsilon m^n+Cm^nt^{-n}e^{\frac{-\varepsilon_0m \hat{d}(\pi(\widetilde{x}),X_{{\rm sing\,}})^2}{t}},\ \ \forall (t, \widetilde{x})\in \, \mathbb{R}_+ \times \widetilde{X}_{{\rm reg\,}}.
	\end{split}
	\end{equation}
\end{theorem}

Recall that since $\Gamma$ acts on $\widetilde{X}$ freely so that $\widetilde{X}/\Gamma = X$, hence, we have $\widetilde{X}_{{\rm reg\,}}/\Gamma \, = \, X_{{\rm reg\,}}$.

%%%%%%%%%%%%%%%%%%%%%%%%%%%%%%%%%%%%%%%%%%%%%%%%%%%%%%%%%%%%%%%%%%%%%%%%%%%

\section{Heat kernel proof}\label{S:proof}
In this section, we will present the heat kernel proof of the main theorem.

We denote by $\operatorname{Tr}_{\Gamma, q}$ the $\Gamma$-trace of operators acting on $L^2_m ( \widetilde{X}, T^{*0,q} \widetilde{X} \otimes \widetilde{E})$, see Subsection~\ref{SS:covering} or \cite[Subsection 3.6.1]{MM}.
\begin{lemma}\label{l:3.6.6}
For any $t>0, m\in \N, 0\leq q\leq n$, we have
\begin{equation}\label{E:3.6.20}
\sum_{j=0}^{q}(-1)^{q-j}\dim_\Gamma \overline{H}_{b, (2), m}^{j}(\widetilde{X},\widetilde{E})\leq \sum_{j=0}^{q}(-1)^{q-j}
{\rm Tr}_{\Gamma, j}[\exp(-\frac{t}{m}\widetilde{\Box}_{b,m})],
\end{equation}
with equality for $q=n$.
\end{lemma}
\begin{proof}
Let $E^{j,m}_\lambda$ be the spectral resolution of $\widetilde{\Box}_{b,m}$ acting on $L^2_m ( \widetilde{X}, T^{*0,q} \widetilde{X} \otimes \widetilde{E})$. We consider the projectors $E^{j,m}(]\lambda_1, \lambda_2]) = E^{j,m}_{\lambda_2}-E^{j,m}_{\lambda_1}$, where $\lambda_2 > \lambda_1 \ge 0.$ Then, by the Hodge decomposition \eqref{E:weakHodge}, $\sum_{j=0}^q(-1)^{q-j}E^{j,m}(]\lambda_1, \lambda_2])$ is the projection on the range of $\overline{\partial}_{b,m}E^{q,m}(]\lambda_1, \lambda_2])$ and thus a positive operator. Hence the $\Gamma$-invariant measure $\sum_{j=0}^q(-1)^{q-j}dE^{j,m}_{\lambda}$ is positive on $\left\{ \lambda > 0 \right\}$.  It follows that
\begin{equation}\label{E:3.6.21}
R\, := \, \int_{\lambda >0} e^{-\frac{t}{m}\lambda} \sum_{j=0}^q(-1)^{q-j}dE^{j,m}_{\lambda} \ge 0, 
\end{equation}
and $R$ commutes with $\Gamma$. On the other hand,
\begin{equation}\label{E:3.6.22}
\operatorname{Tr}_{\Gamma, j}\big[ \exp(-\frac{t}{m}\widetilde{\Box}_{b,m}) \big] \, = \, \dim_\Gamma \overline{H}^{j}_{b, (2), m}(\widetilde{X},\widetilde{E}) + \operatorname{Tr}_\Gamma \int_{\lambda > 0} e^{-\frac{t}{m}\lambda} dE^{j,m}_\lambda.
\end{equation}
By \eqref{E:3.6.21} and \eqref{E:3.6.22}, we obtain the result.
\end{proof}

Let Tr$_{q}[\exp(-\frac{t}{m}\Box_{b,m})]$ be the trace of the operator $\exp(-\frac{t}{m}\Box_{b,m})$ acting on $\Omega_{m}^{0,q}(X, E)$.
It is well-known that (see Theorem 8.10 in \cite{R})
\begin{equation}\label{e-02}
{\rm Tr}_{q}[\exp(-\frac{t}{m}\Box_{b,m})]=\int_{X}
{\rm Tr}_{q}[\exp(-\frac{t}{m}\Box_{b,m})(x,x)]dv_{X}(x).
\end{equation}

By \cite[(3.6.7)]{MM} and \cite[(3.6.8)]{MM}, as in \eqref{e-02}, 
\begin{proposition}\label{p-01c}
We have
\begin{equation}\label{E:3.6.23}
\operatorname{Tr}_{\Gamma, q} \Big[ \exp(-\frac{t}{m}\widetilde{\Box}_{b,m}) \Big] \, = \, \int_U \operatorname{Tr}_q \Big[ e^{-\frac{t}{m} \widetilde{\Box}_{b,m}}(\widetilde{x}, \widetilde{x}) \Big]dv_{\widetilde{X}}(\widetilde{x}).
\end{equation}
\end{proposition}

%%%%%%%%%%%%%%%%%%%%%%%%%%%%%%%%%%%%%%%%%%%%%%%%%%%%%%%%%%%%%%%%%%%%%%%%%%%%%%

Now we are in a position to give the heat kernel proof of the Morse inequalities for the Fourier components of reduced $L^2$ Kohn-Rossi cohomology.
\begin{proof}[Proof of Theorem \ref{t-01covering}]
Denote by Tr$_{\Lambda^{0,q}}$ the trace on $T^{*0,q}X$. The basis for $T^{*0,q}X$ is 
\begin{equation}\label{e-11}
\{\overline\omega^{j_{1}} \wedge \cdots\wedge \overline\omega^{j_{q}}:j_{1}<\cdots<j_{q} \}.
\end{equation}
We write for the index $(1,...,q)$
\begin{equation}\label{e-12}
\begin{split}
&\exp(t\gamma_{d})(\overline\omega^1 \wedge \cdots\wedge \overline\omega^q)\\
&=\prod_{j=1}^{q}(1+(e^{-ta_{j}}-1)\overline\omega^j\wedge\iota_{\overline\omega^j} )
(\overline\omega^1 \wedge \cdots\wedge \overline\omega^q)\\
&=\sum_{k_{1}<\cdots<k_{q}}c_{k_{1}...k_{q}}(x)\overline\omega^{k_{1}} \wedge \cdots\wedge \overline\omega^{k_{q}}.
\end{split}
\end{equation}
From direct calculations, we see that
\begin{equation}\label{e-13}
c_{1...q}(x)=\exp(-t\sum_{j=1}^{q}a_{j}(x)).
\end{equation}
Then we have
\begin{equation}\label{e-14}
{\rm Tr}_{\Lambda^{0,q}}[\exp(t\gamma_{d})]=\sum_{j_{1}<\cdots<j_{q}}
\exp(-t\sum_{i=1}^{q}a_{j_{i}}(x)).
\end{equation}
%Since $X$ is compact and strongly pseudoconvex, $a_{j}(x)>0$ uniformly on $X$ for %$1\leq j\leq n$.
Hence
\begin{equation}\label{e-15}
\begin{split}
&\lim_{t\rightarrow\infty}\frac{{\rm Tr}_{\Lambda^{0,q}}[\exp(t\gamma_{d})]  }{\det(1-\exp(-t\dot{\mathcal{R}}))}\\
&=\lim_{t\rightarrow\infty}\frac{\sum_{j_{1}<\cdots<j_{q}}\exp(-t\sum_{i=1}^{q}a_{j_{i}}(x))}{\prod_{j=1}^{n}(1-\exp(-ta_{j}(x)))  }=(-1)^{q}1_{X(q)},
%&=\begin{cases}
%1& q=0\\
%0& q>0.
%\end{cases}
\end{split}
\end{equation}
where the function $X(q)$ is defined by $1$ on $X(q)$, $0$ otherwise.
As usual, for $\widetilde{x} \in \widetilde{X}$, $\pi(\widetilde{x})=x \in X$. It follows from Theorem \ref{C:3.6.5}, \eqref{E:3.6.23} and Lemma \ref{l:3.6.6} that
\begin{equation}\label{e-16}
\begin{split}
&\frac{1}{m^{n}}\sum_{j=0}^{q}(-1)^{q-j}\dim_{\Gamma}\overline{H}^{j}_{b, (2), m}(\widetilde{X},\widetilde{E})\\ 
&\leq \frac{1}{m^{n}}\sum_{j=0}^{q}(-1)^{q-j} {\rm Tr}_{\Gamma, q}[\exp(-\frac{t}{m}\widetilde{\Box}_{b,m})]\\
&=  \frac{1}{m^{n}}\sum_{j=0}^{q}(-1)^{q-j}   \int_{U}{\rm Tr}_{\Gamma, q}[\exp(-\frac{t}{m}\widetilde{\Box}_{b,m}(\widetilde{x},\widetilde{x}))]dv_{\widetilde{X}}(\widetilde{x})\\
&\leq  (2\pi)^{-n-1}\sum\limits^{p}_{s=1}e^{\frac{2\pi(s-1)}{p}mi}\sum_{j=0}^{q}(-1)^{q-j}
\int_{X}\frac{\det(\dot{\mathcal{R}}) 
{\rm Tr}_{\Lambda^{0,q}}[\exp(t \gamma_d)\otimes \operatorname{Id}_{E_x}]}{\det(1-\exp(-t\dot{\mathcal{R}}))} dv_{X}(x)\\
&+\epsilon \sum_{j=0}^{q}(-1)^{q-j}\text{Vol}(X)+C\sum_{j=0}^{q}(-1)^{q-j}\int_{X}t^{-n}e^{\frac{-\varepsilon_0m \hat{d}(x_0,X_{{\rm sing\,}})^2}{t}}  dv_{X}(x).
\end{split}
\end{equation}
Note that $\epsilon$ is arbitrarily small. By the dominant convergence theorem with $t\rightarrow \infty$,
we have
\begin{equation}\label{e-17}
\begin{split}
\limsup_{m\rightarrow\infty,p|m}&\frac{1}{m^{n}}\sum_{j=0}^{q}(-1)^{q-j}
\dim_\Gamma \overline{H}^{j}_{b, (2), m}(\widetilde{X},\widetilde{E})
\leq \frac{pr}{(2\pi)^{n+1}}\sum_{j=0}^{q}(-1)^{q-j}
\int_{X(j)}|\det(\dot{\mathcal{R}})| dv_{X}(x),\\
\limsup_{m\rightarrow\infty}&\frac{1}{m^{n}}\sum_{j=0}^{q}(-1)^{q-j}\dim_\Gamma \overline{H}^{j}_{b, (2), m}(\widetilde{X},\widetilde{E})=0, \ \ \text{for} \ \  p\nmid m.
\end{split}
\end{equation}
From Definition \ref{d-gue150508f}, \eqref{E:1.5.15} and \eqref{e-17}, we finally get
\begin{equation}\label{e-18}
\begin{split}
&\sum_{j=0}^{q}(-1)^{q-j}\dim_\Gamma \overline{H}^{j}_{b, (2), m}(\widetilde{X},\widetilde{E})
\leq 
\frac{prm^{n}}{2\pi^{n+1}}\sum_{j=0}^{q}(-1)^{q-j}\int_{X(j)}|\det(\mathcal{L}_{x})| dv_{X}(x)+o(m^{n}),
\ \ \text{for} \ \ p|m, \\
&\sum_{j=0}^{q}(-1)^{q-j}\dim_\Gamma \overline{H}^{j}_{b, (2), m}(\widetilde{X},\widetilde{E})=o(m^{n}), \ \ \text{for} \ \  p\nmid m.
\end{split}
\end{equation}
Let $q=n$ in \eqref{E:3.6.20}, by applying Theorem \ref{C:3.6.5}, we obtain for $p|m$,
\begin{equation}\label{e-19}
\begin{split}
&\frac{1}{m^{n}}\sum_{j=0}^{n}(-1)^{n-j}\dim_\Gamma \overline{H}^{j}_{b, (2), m}(\widetilde{X},\widetilde{E})\\
%\sum_{j=0}^{n}(-1)^{n-j}\dim \overline{H}^{j}_{(2), m}(\widetilde{X},\widetilde{E})
&\geq \frac{1}{m^{n}}\sum_{j=0}^{n}(-1)^{n-j}\int_{U}
{\rm Tr}_{\Gamma, j}[\exp(-\frac{t}{m}\widetilde{\Box}_{b,m}(\widetilde{x},\widetilde{x}))]dv_{\widetilde{X}}(\widetilde{x})\\
&\geq  (2\pi)^{-n-1}p
\sum_{j=0}^{n}(-1)^{n-j}\int_{X}\frac{\det(\dot{\mathcal{R}}) 
	{\rm Tr}_{\Lambda^{0,j}}[\exp(t \gamma_d)\otimes \operatorname{Id}_{E_x}]}{\det(1-\exp(-t\dot{\mathcal{R}}))} dv_{X}(x)\\
&-\epsilon n\text{Vol}(X)-Cn\int_{X}t^{-n}e^{\frac{-\varepsilon_0m \hat{d}(x_0,X_{{\rm sing\,}})^2}{t}}  dv_{X}(x).
\end{split}
\end{equation}
Note that $\epsilon$ is arbitrarily small. By the dominant convergence theorem with $t\rightarrow \infty$,
we have
\begin{equation}\label{e-20}
\liminf_{m\rightarrow\infty,p|m}\frac{1}{m^{n}}\sum_{j=0}^{n}(-1)^{n-j}\dim_\Gamma \overline{H}^{j}_{b, (2), m}(\widetilde{X},\widetilde{E})
\geq 
\frac{pr}{(2\pi)^{n+1}}\sum_{j=0}^{n}(-1)^{n-j}\int_{X(j)}|\det(\dot{\mathcal{R}})| dv_{X}(x).
\end{equation}
Then
\begin{equation}\label{e-21}
\liminf_{m\rightarrow\infty,p|m}\frac{1}{m^{n}}\sum_{j=0}^{n}(-1)^{n-j}\dim_\Gamma \overline{H}^{j}_{b, (2), m}(\widetilde{X},\widetilde{E})
=
\frac{pr}{(2\pi)^{n+1}}\sum_{j=0}^{n}(-1)^{n-j}\int_{X(j)}|\det(\dot{\mathcal{R}}) |dv_{X}(x).
\end{equation}
We finally get
\begin{equation}\label{e-22}
\sum_{j=0}^{n}(-1)^{n-j}\dim_\Gamma \overline{H}^{j}_{b, (2), m}(\widetilde{X},\widetilde{E})
=\frac{prm^{n}}{2\pi^{n+1}}\sum_{j=0}^{n}(-1)^{n-j}\int_{X(j)}|\det(\mathcal{L}_{x})| dv_{X}(x)+o(m^{n}) \ \ \text{for} \ \ p|m.
\end{equation}
Then the proof is completed.
\end{proof}

%%%%%%%%%%%%%%%%%%%%%%%%%%%%%%%%%%%%%%%%%%%%%%%%%%%%%%%%%%%%%%%%%%%%%%%%%%%%%%%%%%%%%

\bibliographystyle{plain}

\end{document}